%% file: main.tex
\documentclass{amsart}

\input{header}

\begin{document}

% catcodes

%%%%%%%%%%%%%%%%%%%%%%%%%%%%%%%%%%%%%%%%%%%%%%%%%%%%%%%%%%%%%%%%%%%%%%%%%%%%%%%%%%%%%%

\expandafter\title[%
    Rigidity, tensegrity and reconstruction of polytopes
]{%
    Rigidity, tensegrity and reconstruction of polytopes under metric constraints
%Rigidity and geometric expansion of polytope skeleta and other embedded graphs
%{Universal Rigidity and geoemtric expansion of Polytope Skeleta and other graph embeddings% under edge length constraints
% {Universal Rigidity of Polytope Skeleta
% and graph embeddings under edge length constraints
% {Rigidity and Uniqueness of Polytope Skeleta
% and graph embeddings under length constraints
} % and Related Embeddings}
%{Geometrically Expanded Graph Embeddings} %and the Colin de Verdière Graph Invariant}
		
%\author[M. H. Winter]{Martin H Winter}
\author[M. Winter]{Martin Winter}
\address{Mathematics Institute, University of Warwick, Coventry CV4 7AL, United Kingdom}
\email{martin.h.winter@warwick.ac.uk}
	
% 51M20 - Polyhedra and polytopes
% 52B11 - $n$-dimensional polytopes
% 52C25 - Rigidity and flexibility of structures
%
% 52B05 - Combinatorial properties of polytopes and polyhedra
% 52B12 - Special polytopes
% 52B15 - Symmetry properties of polytopes
% 05C50 - Graphs and linear algebra
% 05C62 - Graph representations
		
\date{\today}
\begin{abstract}
We conjecture that a convex polytope is uniquely determined up to isometry by its edge-graph, edge lengths and the collection of distances~of its vertices to some arbitrary interior point, across all dimensions and all combinatorial types. 
We conjecture even stronger that for two polytopes $P\subset\RR^d$ and $Q\subset\RR^e$ with the same edge-graph it is not possible that $Q$ has longer edges than $P$ while also having smaller vertex-point distances.

We develop techniques to attack these questions and we verify them in~three relevant special cases: $P$ and $Q$ are centrally symmetric, $Q$ is a slight perturbation of $P$, and $P$ and $Q$ are combinatorially equivalent. 
In the first two cases the statements stay true if we replace $Q$ by some graph embedding $q\:V(G_P)\to\RR^e$ of the edge-graph $G_P$ of $P$, which can be interpreted as local \resp\ universal rigidity of certain tensegrity frameworks.
We also establish that a polytope is uniquely determined up to affine equivalence by its edge-graph, edge lengths and the Wachspress coordinates of an arbitrary interior point.

We close with a broad overview of related and subsequent questions.

\end{abstract}

\maketitle

\input{sec/introduction}

\par\bigskip
\parindent 0pt
\textbf{Funding.} 
This work was supported by the Engineering and Physical Sciences Research Council [EP/V009044/1]

\par\bigskip
\parindent 0pt
\textbf{Acknowledgements.} 
I thank Raman Sanyal, Joseph Doolittle, Miek Messerschmidt, Bernd Schulze, James Cruickshank, Robert Connelly and Albert Zhang for many fruitful discussions on the topic of this article, many of which lead to completely new perspectives on the results and to numerous subsequent questions.
%\msays{Raman Sanyal, Joseph Doolittle, Miek Messerschmidt, Bernd Schulze, James Cruickshank ...}
%The author gratefully acknowledges the support by the funding of the European Union and the Free State of Saxony (ESF).

%%%%%%%%%%%%%%%%%%%%%%%%%%%%%%%%%%%%%%%%%%%%%%%%%%%%%%%%%%%%%%%%%%%%%%%%%%%%%%%%%%

\bibliographystyle{abbrv}
\bibliography{literature}

\end{document}

%% file: header.tex
\usepackage[T1]{fontenc}
\usepackage[utf8]{inputenc}
\usepackage[USenglish]{babel}
\usepackage{textcase}

\usepackage[
	bookmarks=true,
	plainpages=false,
	linktocpage,
	colorlinks=true,
	citecolor=green!80!black,
	linkcolor=red!70!black,
	filecolor=magenta,
	urlcolor=magenta,
	breaklinks,
	pdfauthor={Martin Winter},
]{hyperref}

\usepackage{amsmath,amsthm}
\usepackage{calc, mathtools} %calc for \widthof, mathtools for \mathclap

\iftrue
\usepackage{amssymb} % incompatible with mathdesign
\else
\usepackage[charter,cal=cmcal]{mathdesign}
\fi

\usepackage{colortbl,color} % colors
\usepackage[dvipsnames]{xcolor}
\usepackage{centernot} % crossing out symbols correctly
\usepackage{array} % arrays
\usepackage{enumitem,moreenum} % numbered lists where each item can be labeled
\usepackage{cite} % for grouping references in the text which are adjacent in the bibliography
\usepackage[nameinlink,capitalize,noabbrev]{cleveref}
\usepackage{nicefrac}

\usepackage{tikz-cd} 

\usepackage[font=small,labelfont=bf]{caption}

\usepackage{blkarray}

\usepackage{inconsolata}

\newcommand{\ZZ}{\mathbb{Z}}    % integers                             
\newcommand{\RR}{\mathbb{R}}    % real numbers                      
\newcommand{\F}{\mathcal{F}}    % face lattice  
                   
\newcommand{\nozero}{\setminus\{0\}}

\newcommand{\Tsymb}{\top}
\newcommand{\T}{^{\Tsymb}}

\def\^#1{^{(#1)}}
\def\s^#1{^{\smash{(#1)}}}

\newcommand*{\horzbar}{\rule[.5ex]{2.5ex}{0.4pt}}

\newcommand{\wideword}[1]{\quad\text{#1}\quad}
\newcommand{\wideand}{\wideword{and}}

\def\:{\colon}

\newcommand{\cupdot}{\mathbin{\mathaccent\cdot\cup}}

\newcommand{\labelstyle}[1]{\upshape(\textit{#1})}
\newcommand{\mylabel}{\labelstyle{\roman*}}
\newenvironment{myenumerate}{\begin{enumerate}[label=\mylabel]}{\end{enumerate}}

\def\itm#1{{\labelstyle{\romannumeral#1\relax}}}

\def\nlspace{\nolinebreak\space}
\def\nls{\nlspace}

\newcommand{\freespace}{\kern.07em} 
\newcommand{\free}{\freespace\cdot\freespace} 

\newcommand{\enquote}[1]{``#1''}                                 % quotation marks

\newcommand{\ul}[1]{\underline{\smash{#1}}}

\newcommand{\msays}[1]{{\footnotesize\textcolor{red}{\textbf{M:} #1}}}
\newcommand{\TODO}{\msays{TODO}}

\numberwithin{equation}{section}

\newtheoremstyle{mythmstyle} % name
    {\parsep}                    % Space above
    {\parsep}                    % Space below
    {\itshape}                   % Body font
    {}                           % Indent amount
    {\bfseries\scshape}          % Theorem head font
    {.}                          % Punctuation after theorem head
    {.5em}                       % Space after theorem head
    {}  % Theorem head spec (can be left empty, meaning ‘normal’)
\newtheoremstyle{mydefstyle} % name
    {\parsep}                    % Space above
    {\parsep}                    % Space below
    {}                   % Body font
    {}                           % Indent amount
    {\mdseries\scshape}          % Theorem head font
    {.}                          % Punctuation after theorem head
    {.5em}                       % Space after theorem head
    {}  % Theorem head spec (can be left empty, meaning ‘normal’)

\theoremstyle{theorem}
\newtheorem{theorem}{Theorem}[section]
\newtheorem{corollary}[theorem]{Corollary}
\newtheorem{lemma}[theorem]{Lemma}
\newtheorem{proposition}[theorem]{Proposition}
\newtheorem{conjecture}[theorem]{Conjecture}

\theoremstyle{definition}
\newtheorem{definition}[theorem]{Definition}
\newtheorem{example}[theorem]{Example}
\newtheorem{remark}[theorem]{Remark}
\newtheorem{question}[theorem]{Question}
\newtheorem{observation}[theorem]{Observation}
\newtheorem{construction}[theorem]{Construction}
\newtheorem{statement}[theorem]{Statement}

\crefname{theorem}{Theorem}{Theorems}
\crefname{proposition}{Proposition}{Propositions}
\crefname{lemma}{Lemma}{Lemmas}
\crefname{corollary}{Corollary}{Corollaries}
\crefname{remark}{Remark}{Remarks}
\crefname{example}{Example}{Examples}
\crefname{definition}{Definition}{Definitions}
\crefname{problem}{Problem}{Problems}
\crefname{observation}{Observation}{Observation}
\crefname{construction}{Construction}{Construction}

\theoremstyle{theorem}
\newtheorem{innercustomgeneric}{\customgenericname}
\providecommand{\customgenericname}{}
\newcommand{\newcustomtheorem}[2]{%
  \newenvironment{#1}[1]
  {%
   \renewcommand\customgenericname{#2}%
   \renewcommand\theinnercustomgeneric{##1}%
   \innercustomgeneric
  }
  {\endinnercustomgeneric}
}

\newcustomtheorem{theoremX}{Theorem}
\newcustomtheorem{lemmaX}{Lemma}
\newcustomtheorem{conjectureX}{Conjecture}

\DeclareMathOperator{\aff}{aff}
\DeclareMathOperator{\conv}{conv}
\DeclareMathOperator{\cone}{cone}
\DeclareMathOperator{\Aut}{Aut}

\DeclareMathOperator{\diag}{diag}
\DeclareMathOperator{\Ortho}{O}

\DeclareMathOperator{\Span}{span}
\DeclareMathOperator{\im}{im}

\DeclareMathOperator{\rank}{rank}
\DeclareMathOperator{\tr}{tr}

\DeclareMathOperator{\Id}{Id}
\DeclareMathOperator{\id}{id}

\DeclareMathOperator{\dist}{dist}

\DeclareMathOperator{\vol}{vol}  	
\DeclareMathOperator{\Int}{int}  	
\DeclareMathOperator{\relint}{relint}

\DeclareMathOperator{\skel}{sk}

\let\eps=\epsilon

\let\<=\langle
\let\>=\rangle
\let\x=\times

\def\...{...}
\newcommand{\shortStyle}{\textit}
\newcommand{\ie}{\shortStyle{i.e.,}}

\newcommand{\eg}{\shortStyle{e.g.}}

\newcommand{\wrt}{\shortStyle{w.r.t.}}

\newcommand{\cf}{\shortStyle{cf.}}

\newcommand{\resp}{resp.}

\let\angle=\measuredangle
\makeatletter
\renewcommand*{\eqref}[1]{%
  \hyperref[{#1}]{\textup{\tagform@{\ref*{#1}}}}%
}
\makeatother

\newenvironment{disable}
    {\color{lightgray}}
    {}

%% file: sec/introduction.tex
\section{Introduction}
\label{sec:introduction}

In how far can a convex polytope be reconstructed from partial combinatorial~and geometric data, such as its edge-graph, edge lengths, dihedral angles, etc., optionally up to combinatorial type, affine equivalence, or even isometry?
Questions of~this~na\-ture have a long history 
%(one only thinks of Cauchy's rigidity theorem, see \cite{aigner2010proofs,pak2010lectures}; or of Alexandrov's whole book ``Convex Polyhedra'' \cite{alexandrov2005convex})
and are intimately linked~to~the various notions of rigidity.

In this article we address the reconstruction from the edge-graph and some~``graph-compatible'' distance data, such as edge lengths.
It is well-understood that the~edge-graph alone carries very little information about the polytope's full combinatorics, and trying to fix this by supplementing 
%it with 
additional metric data reveals 
%that reconstructing either geometry or combinatorics works best for opposing classes of polytopes.
two opposing effects at play. % that are at play.

%If one tries to reconstruct a convex polytope up to isometry from only its edge-graph together with some ``graph-compatible'' metric data (such as edge lengths), one~observes that two oppo\-sing effects are at play.

% The reconstruct of the geometry of a polytope from, say, edge lengths is made easier (or possible in the first place) by the presence of ``many'' edges.
% For example, simplicial polytopes are determined, up to isometry, by their combinatorial type and their edge lengths already by Cauchy's rigidity theorem. 
% In contrast, if polytopes are not simple, they might be ``flexible'', which is most evident for polygons, but happens non-trivially in higher dimensions as well (see \cref{fig:flex_polytope}).

% However, since we

First and foremost, we need to reconstruct the full combinatorics.
As a general rule of thumb, reconstruction from the edge-graph appears more tractable for polytope that have relatively few edges (such as \emph{simple} polytopes as proven by Blind \& Mani \cite{blind1987puzzles} and later by Kalai \cite{kalai1988simple})\footnote{Though ``few edges'' is not the best way to capture this in general, see \cite{doolittle2017reconstructing} or \cite{joswig2000neighborly}.}.
\mbox{At the same time~however such} polytopes often have too few edges to encode sufficient metric data for~reconstructing the geometry.
This is most~evident~for polygons, but happens non-trivially in higher~di\-mensions and with non-simple polytopes as well (see~\cref{fig:flex_polytope}).

\begin{figure}[h!]
    \centering
    \includegraphics[width=0.62\textwidth]{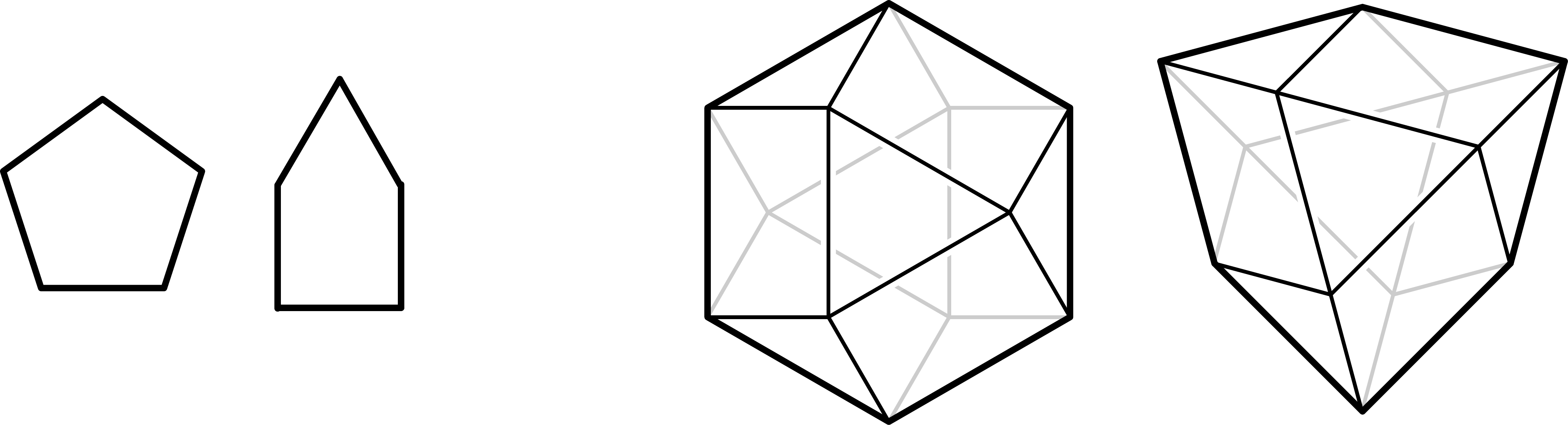}
    \caption{Non-isometric realizations with the same edge lenghts.}
    \label{fig:flex_polytope}
\end{figure}

In contrast, \emph{simplicial} polytopes have many edges and it follows from Cauchy's rigidity theorem that such are determined up to isometry from their edge lengths; \emph{if} we assume knowledge of the full combinatorics.
For simplicial polytopes however, the edge-graph alone is usually not enough to reconstruct the combinatorics in the first place (as evidenced by the abundance of neighborly polytopes).

%\begin{minipage}{\textwidth}
This leads to the following question: how much and what kind of data do we~need to supplement to the edge-graph to permit
\begin{myenumerate}
    \item unique reconstruction of the combinatorics, also for polytopes with many edges (such as simplicial polytopes), and at the same time,
    \item unique reconstruction of the geometry, also for polytopes with few edges (such as simple polytopes).
\end{myenumerate}
Also, ideally the supplemented data fits into the structural framework provided by the edge-graph, that is, contains on the order of $\#\text{edges}+\#\text{vertices}$ datums.

We propose the following: besides the edge-graph and edge lengths, we also fix a point in the interior of the polytope $P$, and we record its distance to each vertex of $P$ (\cf\ \cref{fig:central_rods}).
We believe that this is sufficient data to reconstruct the polytope up to isometry across all dimensions and all combinatorial types.

%While edge-graph and edge lengths are not sufficient, we believe that we can have a unique reconstruction with the following modification: fix a point in the interior of the polytope $P$, and we record its distances to the vertices of $P$ (\cf\ \cref{fig:central_rods}).

%We believe however that we can have a unique reconstruction with the following modification: in addition to the edge lengths, we also fix a point in the interior~of~the polytope $P$, and we record its distances to the vertices of $P$ (\cf\ \cref{fig:central_rods}).

\begin{figure}[h!]
    \centering
    \includegraphics[width=0.37\textwidth]{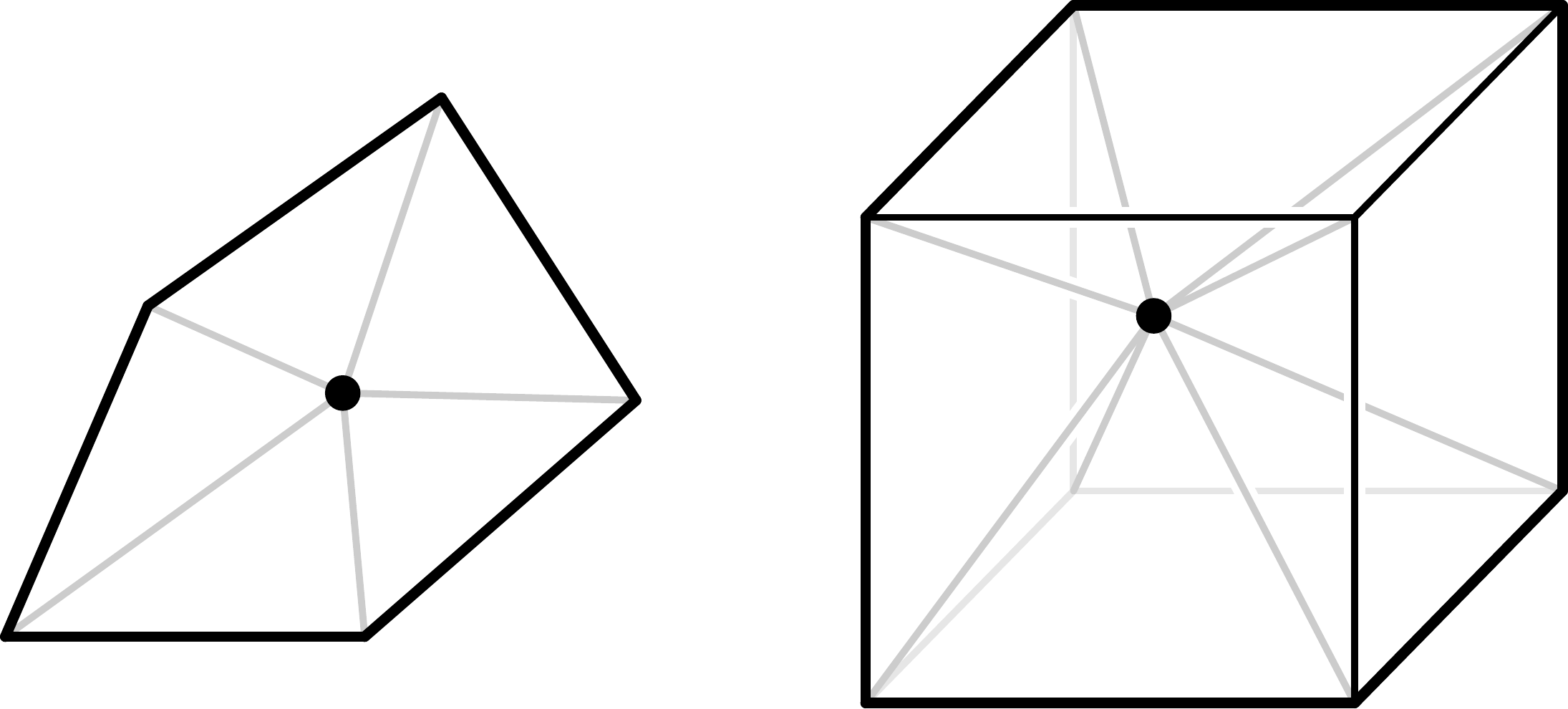} 
    \caption{A ``pointed polytope'', \ie\ a polytope $P\subset\RR^d$ with a point $x\in\Int(P)$. In addition to the edge lengths we also record the lengths~of the gray bars -- the ``vertex-point distances''.}
    \label{fig:central_rods}
\end{figure}

%Indeed, we feel confident enough about this claim to state it as \cref{conj:main_rigid_intro}~below.
Here and in the following we can assume that the polytopes are suitably transla\-ted so that the chosen point is the origin $0\in\Int(P)$.

% \begin{conjecture}
% A polytope with the origin in its interior is uniquely determined, up to isometry, by the edge-graph, edge-lengths and vertex-origin distances.
% \end{conjecture}

% \begin{conjecture}\label{conj:main_rigid_intro}
% Given polytopes $P_1\subset\RR^d$ and $P_2\subset\RR^e$ with $0\in\Int(P_i)$ so that $P_1$ and $P_2$ have isomorphic edge-graphs, corresponding edges are of the same length, and corresponding vertices have the same distance to the origin.
% Then $P\simeq Q$ (\ie\ $P$ and $Q$ are isometric via an orthogonal transformation).
% \end{conjecture}

\begin{conjecture}\label{conj:main_rigid_intro}
Given polytopes $P\subset\RR^d$ and $Q\subset\RR^e$ with the origin in their respective interiors, and so that $P$ and $Q$ have isomorphic edge-graphs,\nlspace corresponding edges are of the same length, and corresponding vertices have the same~\mbox{distance}~to~the origin.
Then $P\simeq Q$ (\ie\ $P$ and $Q$ are isometric via an orthogonal \mbox{transformation}).
\end{conjecture}

% We may state this claim more casually as

% \begin{quote}    
% %\begin{conjecture}
% \itshape A polytope with the origin in its interior is uniquely determined,\nls up to isometry, by its edge-graph, edge-lengths and vertex-origin distan\-ces.
% %\end{conjecture}
% \end{quote}

% Note that we claim this unique reconstruction across \emph{all} dimensions and \emph{all}~combinatorial types.
Requiring the origin to lie \emph{in the interior} is necessary to prevent counterexamples such as the one shown in \cref{fig:origin_outside_tnontriv_ex}.
%This conjecture vastly generalizes~several~known results, such as the Kirszbraun theorem or the reconstruction of matroid base polytopes or highly symmetric~polytopes~(see~\cref{sec:vast_generalization}).
This conjecture vastly generalizes~several~known reconstruction results, such as for matroid base polytopes or highly symmetric~polytopes~(see~\cref{sec:vast_generalization}).
%This conjecture is a vast generalization of known reconstruction results, \eg\ of matroid base polytopes (see also \cref{sec:conj_consequences}.

\begin{figure}[h!]
    \centering
    \includegraphics[width=0.45\textwidth]{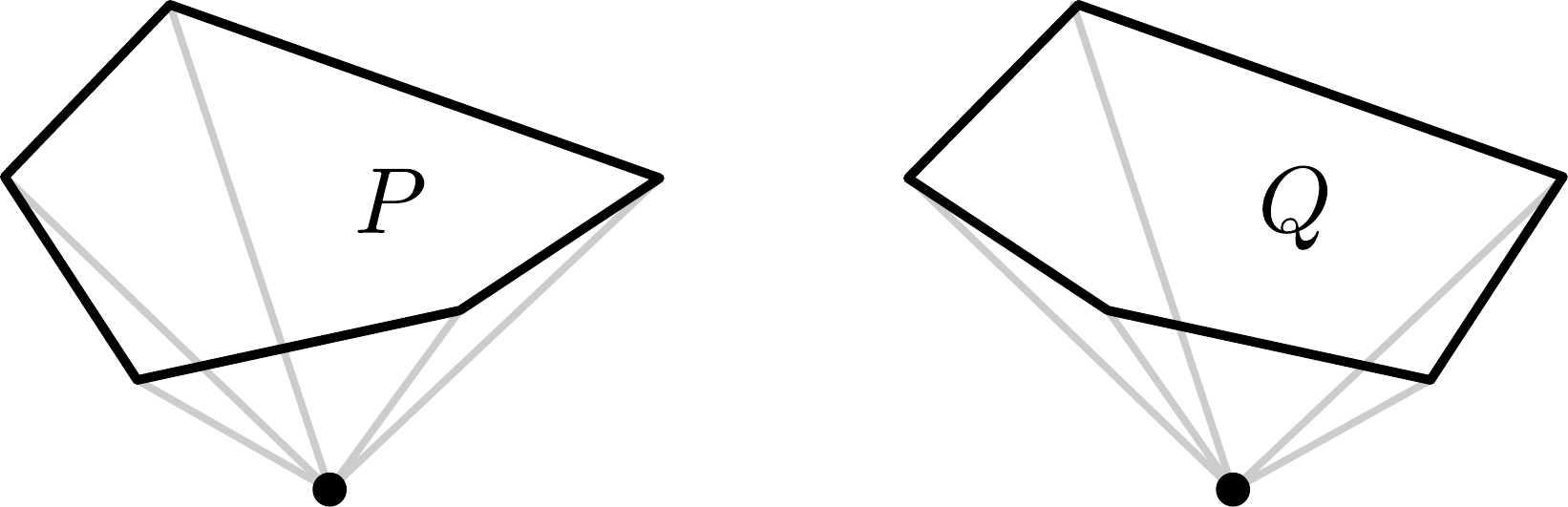}
    \caption{Two non-isometric realizations of a pentagon~with~the~same \mbox{edge lengths and vertex-point distances. This is possible~\mbox{because}~the~point} is not in the interior.}
    \label{fig:origin_outside_tnontriv_ex}
\end{figure}

We also make the following stronger conjecture:

\begin{conjecture}\label{conj:main}
Given two polytopes $P\subset\RR^d$ and $Q\subset\RR^e$ with isomorphic edge-graphs, and so that
\begin{myenumerate}
    \item $0\in\Int(Q)$,
    \item edges in $Q$ are \ul{at most as long} as their counterparts in $P$, and
    \item vertex-origin distances in $Q$ are \ul{at least as large} as their counterparts in $P$,
\end{myenumerate}
then $P\simeq Q$ ($P$ and $Q$ are isometric via an orthogonal transformations).
\end{conjecture}

Intuitively, \cref{conj:main} states that a polytope cannot become larger (or~``more expanded'' as measured in vertex-origin distances) while its edges are getting shorter.
It is clear that \cref{conj:main_rigid_intro} is a consequence of \cref{conj:main}, and we shall call the former the ``unique reconstruction version'' of the latter.
Here, the necessity of the precondition $0\in\Int(Q)$ can be seen even quicker: vertex-origin distances~can~be increased arbitrarily by translating the polytope just far enough away from the origin (see also \cref{fig:origin_outside_trivial_ex}).

\begin{figure}[h!]
    \centering
    \includegraphics[width=0.43\textwidth]{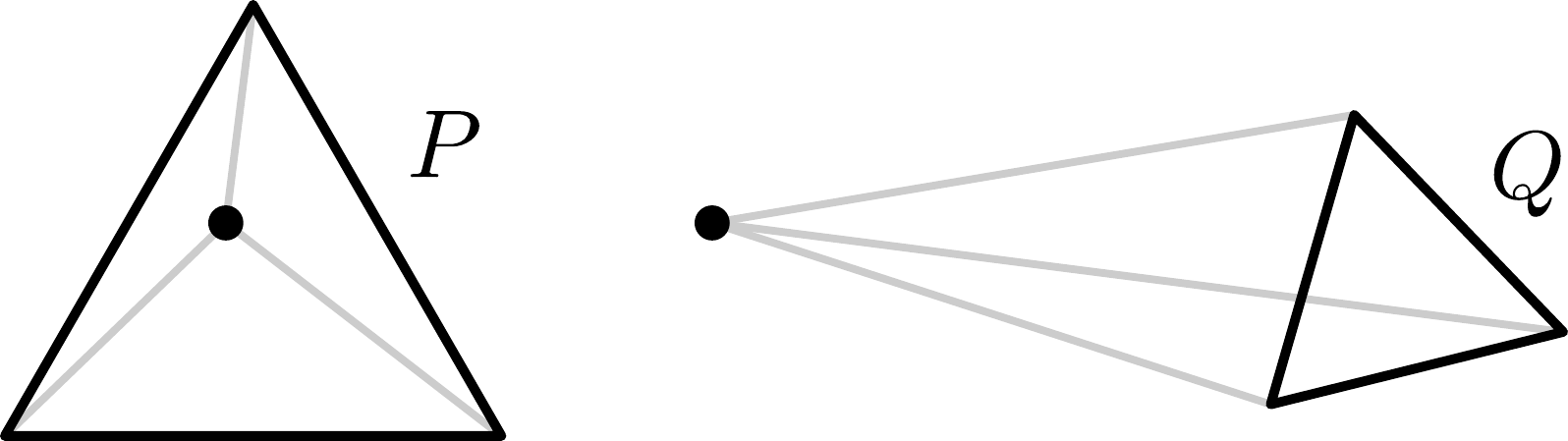}
    \caption{If $x\not\in\Int(Q)$, then it is possible for $Q$ to have shorter edges than $P$, while simultaneously also all vertices farther away from $x$.}
    \label{fig:origin_outside_trivial_ex}
\end{figure}

In this article we develop techniques that we feel confident point us the right~way towards a resolution of the conjectures.
We then verify the conjectures~in~the~follo\-wing three relevant special cases: %, for which we veri\-fy the conjectures:
\begin{itemize}
    \item $P$ and $Q$ are centrally symmetric (\cref{res:centrally_symmetric}),
    \item $Q$ is a slight perturbation of $P$ (\cref{thm:main_local}),
    \item $P$ and $Q$ are combinatorially equivalent (\cref{thm:main_comb_eq}). %This, in particular, includes the case of 3-dimensional polytopes.
\end{itemize}
The last special case clarifies, in particular, the case of 3-dimensional polytopes.\nls
Also, our eventual formulations of the first two special cases will in fact be more general, replacing $Q$ by some embedding $q\:V(G_P)\to\RR^e$ of the edge-graph $G_P$, where~$q$ is no longer assumed to be the skeleton of any polytope.
These results can then~also be interpreted as claiming rigidity, local or universal, of certain bar-joint or tensegrity frameworks.

\subsection{Notation and terminology}
\label{sec:setting}

Throughout the article, all polytopes are convex and bounded, in particular, can be written as the convex hull of their vertices:
$$P=\conv\{p_1,...,p_n\}:=\Big\{\smash{\sum_{i}\alpha_i p_i \mid \alpha\in\Delta_n}\Big\},$$
where $\Delta_n:=\{x\in\RR^n_{\ge 0}\mid x_1+\cdots+x_n=1\}$ denotes the set of convex coefficients.

If not stated otherwise, $P\subset\RR^d$ will denote a polytope in $d$-dimensional space~for $d\ge 2$, though 
%it might not itself be $d$-dimensional, in that 
its affine hull $\aff(P)$ might be a proper subspace of $\RR^d$.
If $\dim\aff(P)=d$ we say that $P$ is \emph{full-dimensional}. 
Our polytopes are often \emph{pointed}, that is, they come with a special point $x\in\Int(P)$ (sometimes also on $\partial P$ or outside); but we usually translate $P$ so that $x$ is the origin.
So, instead of distances from the vertices to $x$, we just speak of \emph{vertex-origin distances}.

By $\F(P)$ we denote the \emph{face-lattice} of $P$, and by $\F_\delta(P)$ the subset of $\delta$-dimensional faces.
We shall assume a fixed enumeration $\F_0(P)=\{p_1,...,p_n\}$ of the polytope's vertices (\ie\ our polytopes are \emph{labelled}), in particular, the number of vertices will be denoted by $n$.
We also often use a polytope $Q\subset\RR^e$ whose vertices are denoted $\F_0(Q)=\{q_1,...,q_n\}$.

The edge-graph of $P$ is the finite simple graph $G_P=(V,E)$, where $V=\{1,...,n\}$ is compatible with the vertex labelling, that is, $i\in V$ corresponds to $p_i\in \F_0(P)$~and $ij\in E$ if and only if $\conv\{p_i,p_j\}\in\F_1(P)$.
The graph embedding given by $i\mapsto p_i$ (with edges embedded as line segments) is called \emph{(1-)skeleton} $\skel(P)$ of $P$.
% (\emph{1-})\emph{skeleton}

When speaking of combinatorially equivalent polytopes $P$ and $Q$, we shall implicitly fix a face-lattice isomorphism $\phi \:\F(P)\xrightarrow{\sim} \F(Q)$ compatible with the vertex labels, \ie\ $\phi(p_i)=q_i$.
This also allows us to implicitly associate faces of $P$ to faces of $Q$, for example, for a face $\sigma\in\F(P)$ we can write $\sigma_Q$ for the corresponding face in $\F(Q)$.
Likewise, if $P$ and $Q$ are said to have isomorphic edge-graphs, we implicitly assume a graph isomorphism $G_P\xrightarrow{\sim}G_Q$ sending $p_i$ onto $q_i$.
We will then often say that $P$ and $Q$ have \emph{a common} edge-graph, say, $G_P$.

We write $P\simeq Q$ to denote that $P$ and $Q$ are isometric. Since our polytopes are usually suitably translated, if not stated otherwise, this isometry can be assumed~as realized by an orthogonal transformation.

Let us repeat \cref{conj:main} using our terminology:

\begin{conjectureX}{\ref{conj:main}}
Given polytopes $P\subset\RR^d$ and $Q\subset\RR^e$ with the same edge-graph $G_P=(V,E)$, so that
\begin{myenumerate}
    \item $0\in\Int(Q)$,
    \item edges in $Q$ as most as long as in $P$, \ie\
    $$\|q_i-q_j\|\le \|p_i-p_j\|,\quad\text{for all $ij\in E$},$$
    \item vertex-origin distances in $Q$ are at least as larger as in $P$, \ie\
    $$\|q_i\|\ge \|p_i\|,\quad\text{for all $i\in V$},$$
\end{myenumerate}
then $P\simeq Q$. %, \ie\ $P$ and $Q$ are isometric via some orthogonal transformation.
\end{conjectureX}

%\msays{Strictly speaking, our equivalence are about pointed polytopes, but we shall not write this all the time.}

\subsection{Structure of the article}
\label{sec:overview}

In \cref{sec:warmup} we prove the instructive special case of \cref{conj:main} where both $P$ and $Q$ are simplices.
While comparatively straightforward, the proof helps us to identify a quantity\nls --\nls we call it the \emph{expansion}~of~a~polytope -- that is at the core of a more general approach.

The goal of \cref{sec:expansion_version} is to show that the \enquote{expansion} of a polytope is monotone in its edge lengths, that is, decreases when the edge lengths shrink. 
In~fact,\nls we~verify this in the more general context that replaces $Q$ by a general embedding $q\: V(G_P)\to\RR^d$ of $P$'s edge-graph.
As a main tool we introduce~the~\emph{Wachspress coordinates} (a special class of generalized barycentric coordinates) and discuss a theo\-rem of Ivan Izmestiev.
%We eventually prove the main theorem about polytope expansion.

In \cref{sec:expansion_rigidity_tensegrity} we apply these results to prove \cref{conj:main} for the three special cases: centrally symmetric, close-by and combinatorially equivalent polytopes. We also discuss the special case of inscribed polytopes.
We elaborate how our tools can potentially be used to attack \cref{conj:main}.

% In \cref{sec:relations} we discuss how our finding relate to other open and recently verified conjectures, such as Stoker's conjecture. We discuss how some philosophical implications and an implication to the dimension of the realization space of a polytope.

In \cref{sec:conclusions} we conclude our investigation with further thoughts on our results, notes on connections to the literature, as well as \emph{many} questions and future research directions.
Despite being a conclusion section, it is quite rich in content, as we found it more appropriate to gather many notes there rather than to repeatedly~interrupt the flow of~the main text.

\iffalse % old introduction

\hrulefill

We make a simple statement:
%
\begin{quote}
    \itshape If all edges of a polytope shorten, the polytope itself necessarily~shrinks in size. Or can it still get bigger?
\end{quote}
%
While this seems plausible, it certainly lacks in detail. Most importantly, what do we even mean by the \emph{size} of a polytope.

\begin{theorem}
Given two combinatorially equivalent polytope $P,Q\subset\RR^d$ so that
%
\begin{myenumerate}
    \item $P$ and $Q$ are inscribed in spheres of radii $r_P$ and $r_Q$ respectively,
    \item $Q$ contains the origin in its interior, and
    \item edges in $Q$ are not longer than edges in $P$,
\end{myenumerate}
%
then $r_P\ge r_Q$.
\end{theorem}

\subsection{Notes on notation}

By $\F(P)$ we denote the \emph{face lattice} of $P$, that is, the set of faces of $P$ ordered by inclusion.
By $\F_\delta(P)$ we denote the subset of $\delta$-dimensional faces.
Two polytopes $P,Q\subset\RR^d$ are said to be \emph{combinatorially equivalent} if their face lattices $\F(P)$ and $\F(Q)$ are isomorphic as lattices.
We use the following~convention: if we speak of combinatorially equivalent polytopes $P,Q\subset\RR^d$, we implicitly fix a face lattice isomorphism $\phi\:\F(P)\to\F(Q)$. For each face $\sigma\in\F(P)$ we then speak of the \emph{corresponding face} $\sigma_Q:=\phi(\sigma)$ in $Q$, and vice versa.

\fi

\section{Warmup: a proof for simplices}
\label{sec:warmup}

To get acquainted with the task we discuss the instructive special case~of \cref{conj:main} where both $P$ and $Q$ are simplices. % (or any other two polytopes with complete edge-graphs).
The proof is reasonably short~but~contains already central ideas for the general case.

\begin{theorem}\label{res:special_case_simplices}
Let $P,Q\subset \RR^d$ be two simplices so that
\begin{myenumerate}
    \item $0\in\Int(Q)$,
    \item edges in $Q$ are at most as long as in $P$, and % \ie
        %$$\|q_i-q_j\|\le\|p_i-p_j\|,\quad\text{for all $i,j\in V(G_P)$},$$
    \item vertex-origin distances in $Q$ are at least as large as in $P$, %\ie
        %$$\|q_i\|\ge \|p_i\|,\quad\text{for all $i\in V(G_P)$},$$
\end{myenumerate}
then $P\simeq Q$. % ($P$ and $Q$ are isometric).
\begin{proof}
By \itm1 we can choose barycentric coordinates $\alpha\in\Int\Delta_n$ for the origin~in~$Q$, that is, $0=\sum_i\alpha_i q_i$.
Consider the following system of equalities and inequalities:
\begin{align}\notag
    \sum_i \alpha_i \|p_i\|^2
      &= \Big\|\sum_i \alpha_i p_i\Big\|^2 \!+ \tfrac12\sum_{i,j} \alpha_i\alpha_j \|p_i-p_j\|^2
    \\[-1.5ex]
    \rotatebox{90}{$\ge$} 
    \qquad
    & 
    \qquad \quad \;\,
    \rotatebox{90}{$\le$}
    \qquad \qquad \qquad \qquad\!
    \rotatebox{90}{$\le$}
    \\[-1.5ex]\notag
    \sum_i \alpha_i \|q_i\|^2
      &= \Big\|\sum_i \alpha_i q_i\Big\|^2 \!+ \tfrac12\sum_{i,j} \alpha_i\alpha_j \|q_i-q_j\|^2
\end{align}
\label{eq:simplex_case}%
The equalities of the first and second row can be verified by rewriting the norms as inner products followed by a straightforward computation. 
The vertical inequalities follow, from left to right, using \itm3, the definition of $\alpha$, and \itm2 respectively.

But considering this system of (in)equalities, we must conclude that all inequalities are actually satisfied with equality.
In particular, equality in the right-most terms yields $\|p_i-p_j\|=\|q_i-q_j\|$ for all $i,j\in V(G_P)$ (here we are using $\alpha_i>0$).\nlspace%
But sets of points with pairwise identical distances are isometric.
%By~standard arguments, two simplices with identical edge lengths are isometric.
\end{proof}
\end{theorem}

Why can't we apply this proof to general polytopes?
The right-most sum in~\eqref{eq:simplex_case} iterates over all vertex pairs and measures, if you will, a weighted average of pairwise vertex distances in $P$.
In simplices each~vertex pair forms an edge, and hence, if all edges decrease in length, this average de\-creases as well.
In general polytopes however, when edge become shorter, some \enquote{non-edge vertex distances} might still~increase, and so the right-most inequality cannot be obtained in the same term-wise fashion. 
In fact, there is no reason to expect that the inequality holds at all.

It should then be surprising to learn that it actually does hold, at least in some controllable circumstances that we explore in the next section.
This will allow us to generalize \cref{res:special_case_simplices} beyond simplices.

\section{$\alpha$-expansion, Wachspress coordinates and the Izmestiev matrix}
\label{sec:expansion_version}

Motivated by the proof of the simplex case (\cref{res:special_case_simplices}) we define the following measure of size for a polytope (or graph embedding $p\: V(G_P)\to\RR^d$):

\begin{definition}
For $\alpha\in\Delta_n$ the \emph{$\alpha$-expansion} of $P$ is
$$\|P\|_\alpha^2 := \tfrac12\sum_{i,j} \alpha_i\alpha_j \|p_i-p_j\|^2.$$
\end{definition}

The sum in the definition iterates over all pairs of vertices and so the $\alpha$-expansion measures a weighted average of vertex distances, in particular, $\|P\|_\alpha$ is a translation invariant measure.
If all pairwise distances between vertices decrease, so does the~$\alpha$-expansion. 

The surprising fact, and main result of this section (\cref{res:expansion_main_result}), is that for a carefully chosen $\alpha\in\Delta_n$ the $\alpha$-expansion decreases already if only the edge lengths decrease, independent of what happens to other vertex distances.
%is monotone already in the \emph{edge lengths} of $P$, even if some non-edge vertex-distances increase in the process.
%In this sense we can say~that a polytope is \enquote{maximally expanded} for their edge lengths.

In fact, this statement holds true in much greater generality and we state it already here (it mentions \emph{Wachspress coordinates} which we define in the next section; one should read this as \enquote{there exist $\alpha\in\Delta_n$ so that ...}):

\begin{theorem}\label{res:expansion_main_result}
Let $P\subset\RR^d$ be a polytope with edge-graph $G_P=(V,E)$ and let~$\alpha\in\Delta_n$ be the Wachspress coordinates of some interior point $x\in\Int(P)$.
If $q\: V$ $\to\RR^e$ is some embedding of $G_P$ whose edges are at most as long as in $P$, then
$$\|P\|_\alpha\ge \|q\|_\alpha,$$
with equality if and only if $q$ is an affine transformation of the skeleton $\skel(P)$,\nlspace all~ed\-ges of which are of the same length as in $P$.
\end{theorem}

Indeed, \cref{res:expansion_main_result} is not so much about comparing $P$ with another polytope, but actually about comparing the skeleton $\skel(P)$ with some other graph embedding $q$ that might not be the skeleton of any polytope and might even be embedded in a lower- or higher-dimensional Euclidean space.
Morally, \cref{res:expansion_main_result} says: \emph{polytope skeleta are~maximally expanded for their edge lengths}, where \enquote{expansion}~here measures~an average of vertex distances with carefully chosen weights. %measured \wrt\ carefully chosen coefficients $\alpha$.

% Note that \cref{res:expansion_main_result} allows us to \enquote{deform} $\skel(P)$ into an embedding $q\:G_P\to\RR^e$ that is not necessarily a skeleton itself, and that might even live in a lower- or higher-dimensional Euclidean space.

% However, there is a caveat: in contrast to circumradii, the $\alpha$-expansion of a polytope depends on the parameter $\alpha\in\Delta_n$, and it is not a priorly clear whether we should expect our result for just any choice of this parameter.
% In fact, it does not, and special parameters must be chosen.
% We nee to diverge briefly to discuss \emph{Wachspress coordinates}.

The result clearly hinges on the existence of these so-called \emph{Wachspress coordinates}, which we introduce now.

\subsection{Wachspress coordinates}
\label{sec:Wachspress_Izmestiev}

In a simplex $P\subset\RR^d$ each point $x\in P$ can~be~expressed as a convex combination of the simplex's vertices in a unique way:
$$x=\sum_i \alpha_i p_i.$$
The coefficients $\alpha\in\Delta_n$ are called the \emph{barycentric coordinates} of $x$ in $P$.

In a general polytope $P\subset\RR^d$ there are usually many ways to express a point~$x\in P$ as a convex combination of the polytope's vertices.
In many applications however it is desirable to have a canonical choice, so to say \enquote{generalized barycentric~coordinates}.
Various such coordinates have been defined (see \cite{floater2015generalized} for an overview), one of them being the so-called \emph{Wachspress coordinates}.
Those were initially defined by Wachspress for polygons \cite{wachspress1975rational}, and later generalized to general polytopes by Warren et al.~\cite{warren1996barycentric,warren2007barycentric}. 
A construction, with a geometric interpretation due to  \cite{ju2005geometric}, is given in \cref{sec:relation_Wachspress_Izmestiev} below.

The relevance of the Wachspress coordinates for our purpose is however not so much in their precise definition, but rather in their relation to a polytope invariant of \enquote{higher rank} that we introduced next. 

\subsection{The Izmestiev matrix}

At the core of our proof of \cref{res:expansion_main_result} is the observation that the Wachspress coordinates are merely a shadow of a \enquote{higher rank} object that we call the \emph{Izmestiev matrix} of $P$; an $(n\x n)$-matrix associated~to~an~$n$-vertex polytope with $0\in\Int(P)$, whose existence and properties in connection with graph skeleta were established by Lovász in dimension three \cite{lovasz2001steinitz}, and by Izmestiev in general dimension \cite{izmestiev2010colin}. 
We summarize the findings:
%See \cref{sec:appendix_Izmestiev_Wachspress} for an explicit definition and some historical notes.
%We summarize the main result of \cite{izmestiev2010colin}:
%For our purpose we can summarize as follows:

% The second core concept for the proof is the \emph{Izmestiev matrix} \cite{izmestiev2010colin}, which exists for any convex polytope that contains the origin in its interior.
% It turns out that there exists a tight relation between this matrix and the Wachspress coordinates.

%In \cite{...} we pointed out a connection between the Wachspress coordinates and the Izmestiev matrix of a polytope.

\begin{theorem}
%[\!{\cite[Theorem 2.4]{izmestiev2010colin}}]
\label{res:Izmestiev}
Given a polytope $P\subset\RR^d$ with $0\in\Int(P)$ and edge-graph~$G_P=(V,$ $E)$,
there exists a symmetric matrix $M\in\RR^{n\x n}$ (the Izmestiev matrix of $P$) with~the following properties:
\begin{myenumerate}
    \item $M_{ij}>0$ if $ij\in E$,
    \item $M_{ij}=0$ if $i\not=j$ and $ij\not\in E$,
    \item $\dim\ker M=d$,
    \item $MX_P=0$, where $X_P\T=(p_1,...,p_n)\in\RR^{d\x n}$, and
    \item $M$ has a unique positive eigenvalue (of multiplicity one).
\end{myenumerate}
\end{theorem}

Izmestiev provided an explicit construction of this matrix that we discuss in~\cref{sec:relation_Wachspress_Izmestiev} below.
Another concise proof of the spectral properties of the Izmestiev~matrix can be found in the appendix of \cite{narayanan2021spectral}.

%For an explicit construction of $M$ and some further historical notes, see \cref{sec:appendix_Izmestiev_Wachspress}.

\begin{observation}\label{res:Izmestiev_observations}
Each of the properties \itm1 to \itm5 of the Izmestiev matrix will~be crucial for proving \cref{res:expansion_main_result} and we shall elaborate on each point below:
\begin{myenumerate}
    \item \cref{res:Izmestiev} \itm1 and \itm2 state that $M$ is some form of generalized adjacency matrix, having non-zero off-diagonal entries if and only if the polytope has an edge between the corresponding vertices.
    Note however that the theorem tells nothing directly about the diagonal entries of $M$.
    \item \cref{res:Izmestiev} \itm3 and \itm4 tell us precisely how the kernel of $M$ looks like, namely, $\ker M=\Span X_P$. The inclusion $\ker M\supseteq \Span X_P$ follows directly from \itm4. But since $P$ has at least one interior point (the origin) it must be a full-dimensional polytope, meaning that $\rank X_P=d$. Comparison of dimensions (via \itm 3) yields the claimed equality.
    \item let $\{\theta_1 > \theta_2 > \cdots > \theta_m\}$ be the spectrum of $M$. \cref{res:Izmestiev} \itm5 then tells us that $\theta_1>0$, $\theta_2=0$ and $\theta_k<0$ for all $k\ge 3$.
    \item $M':= M+\gamma\Id$ is a non-negative matrix if $\gamma>0$ is sufficiently large,\nlspace and is then subject to the \emph{Perron-Frobenius theorem} (see \cref{res:Perron_Frobenius}). Since the edge-graph $G_P$ is connected, the matrix $M'$ is \emph{irreducible}.
    The crucial information provided by the Perron-Frobenius theorem is that $M'$ has an eigenvector $z\in\RR^n$ to its largest eigenvalue (that is, $\theta_1+\gamma$), all entries~of which are positive.
    By an appropriate scaling we can assume $z\in\Int(\Delta_n)$, which is a $\theta_1$-eigenvector to the Izmestiev matrix $M$, and in fact, spans its $\theta_1$-eigenspace.
\end{myenumerate}
\end{observation}

Note that the properties \itm1 to \itm5 in \cref{res:Izmestiev} are invariant under scaling~of $M$ by a \emph{positive} factor.
As we verify in \cref{sec:relation_Wachspress_Izmestiev} below, $\smash{\sum_{i,j}M_{ij}>0}$,\nlspace and so we can fix the convenient normalization $\smash{\sum_{i,j}M_{ij}=1}$.
In fact, with this~normalization in place we can now reveal that the Wachspress coordinates emerge simply as the row sums of $M$:
\begin{equation}
\label{eq:Wachspress_is_Izmestiev_row_sum}
\alpha_i:=\sum_j M_{ij}, \quad\text{for all $i\in\{1,...,n\}$}.
\end{equation}
This connection has previously been observed in \cite[Section 4.2]{ju2005geometric} for 3-dimensional polytopes, and we shall verify the general case in the next section (\cref{res:Wachspress_is_Izmestiev_row_sum}).

\subsection{The relation between Wachspress and Izmestiev}
\label{sec:relation_Wachspress_Izmestiev}

The Wachspress coordinates and the Izmestiev matrix can be defined simultaneously in a rather elegant fashion: given a polytope $P\subset\RR^d$ with $d\ge 2$ and~$0\in\Int(P)$, as well as a vector~$\mathbf c=$ $(c_1,...,c_n)\in\RR^n$, consider the \emph{generalized polar dual} 
$$P^\circ(\mathbf c):=\big\{x\in\RR^d\mid \<x,p_i\>\le c_i\text{ for all $i\in V(G_P)$}\big\}.$$
We have that $P^\circ(\mathbf 1)$ with $\mathbf 1=(1,...,1)$ is the usual polar dual.
The (unnormalized) Wachspress coordinates $\tilde\alpha\in\RR^n$ of the origin and the (unnormalized) Izmestiev~matrix $\tilde M$ $\in$ $\RR^{n\x n}$ emerge as the coefficients in the Taylor expansion~of~the volume of $P^\circ(\mathbf c)$ at $\mathbf c=\mathbf 1$:
%Developing the volume of $P^\cric(\mathbf c)$ in a Taylor series yields
%
\begin{equation}\label{eq:Taylor_definition}
\vol\!\big(P^\circ(\mathbf c)\big)=\vol(P^\circ)+\<\mathbf c-\mathbf 
1,\tilde\alpha\>+\tfrac12(\mathbf c-\mathbf 1)\T\!\tilde M(\mathbf c-\mathbf 1)+\cdots.
\end{equation}
In other words,
\begin{equation}\label{eq:variational_definition}
\tilde\alpha_i := \frac{\partial\vol(P^\circ(\mathbf c))}{\partial c_i}\Big|_{\mathbf c=\mathbf 1}\quad\text{and}\quad
\tilde M_{ij} := \frac{\partial^2\vol(P^\circ(\mathbf c))}{\partial c_i\partial c_j}\Big|_{\mathbf c=\mathbf 1}.
\end{equation}
%
%The respective normalized quantities $\alpha\in\Delta_n$ and $M\in\RR^{n\times n}$ are obtained by~re\-scaling, so the sum of the entries becomes 1.
In this form, one might recognize the~(unnor\-malized) Izmes\-tiev matrix of $P$ as the \emph{Alexandrov ma\-trix} of the polar dual $P^\circ$.

Geometric interpretations for \eqref{eq:variational_definition} were given in \cite[Section 3.3]{ju2005geometric} and \cite[proof of Lemma 2.3]{izmestiev2010colin}: for a vertex $p_i\in \F_0(P)$ let $F_i\in\F_{d-1}(P^\circ)$ be~the corresponding dual~facet.
Likewise, for an edge $e_{ij}\in\F_1(P)$ let $\sigma_{ij}\in\F_{d-2}(P^\circ)$ be the corresponding dual face of codimension 2.
Then
\begin{equation}\label{eq:geometric_definition}
\tilde\alpha_i = \frac{\vol(F_i)}{\|p_i\|}\quad\text{and}\quad \tilde M_{ij}=\frac{\vol(\sigma_{ij})}{\|p_i\|\|p_j\|\sin\sphericalangle(p_i,p_j)}, % \text{ for $i\not=j$},
\end{equation}
where $\vol(F_i)$ and $\vol(\sigma_{ij})$ are to be understood as relative volume.
The expression for $\tilde\alpha$ is (up to a constant factor) the \emph{cone volume} of $F_i$ in $P^\circ$\!, \ie\ the volume~of~the cone with base face $F_i$ and apex at the origin.
As such it is positive, which confirms again~that we can normalize to $\alpha\in\Delta_n$, and we see that $\alpha_i$ measures the fraction~of the cone vo\-lume at $F_i$ in the total volume of $P^\circ$.
That $\smash{\tilde M}$ can be normalized~follows from the next statement, which is a precursor to \eqref{eq:Wachspress_is_Izmestiev_row_sum}:

\begin{proposition}
%Let $P\subset\RR^d$ be a polytope, $\alpha\in\Delta_n$ the Wachspress coordinates of the origin $0\in\Int(P)$, and $M\in\RR^{n\x n}$ its Izmestiev matrix. Then $\sum_j M_{ij} = \alpha_i$.
%
%\quad
%\begin{myenumerate}
    %\item 
    $\sum_j \tilde M_{ij} = (d-1)\tilde \alpha_i$.
    % $\sum_j \tilde M_{ij} = \tilde\alpha_i$.
    %\item 
    %$\sum_{i,j} \tilde M_{ij} > 0$.
%\end{myenumerate}
\end{proposition}

%In par\-ticular, if $\tilde\alpha$ can be normalized, so can $\tilde M$.

\begin{proof}
Observe first that $\vol(P^\circ(\mathbf c))$ is a homogeneous function of degree $d$, \ie\
$$\vol(P^\circ(t\mathbf c)) = \vol(tP^\circ(\mathbf c)) = t^d\vol(P^\circ(\mathbf c))$$
for all $t\ge 0$.
Each derivative $\partial\vol(P^\circ(\mathbf c))/\partial c_i$ is then homogeneous of degree $d-1$.
\emph{Euler's homogeneous function theorem} (\cref{res:Eulers_homogeneous_function_theorem}) yields
$$
\sum_j c_j \frac{\partial^2\vol(P^\circ(\mathbf c))}{\partial c_i\partial c_j} 
= (d-1)  \frac{\partial\vol(P^\circ(\mathbf c))}{\partial c_i}.
$$
Evaluating at $\mathbf c=1$ and using \eqref{eq:variational_definition} yields the claim. %$\sum_j \tilde M_{ij}=(d-1)\tilde \alpha_i$. %, and normalizing both sides gives the claimed identity for $M$ and $\alpha$.
% 
% $$\sum_i c_i \frac{\partial\vol(P^\circ(\mathbf c))}{\partial c_i} = d \vol(P^\circ(\mathbf c)).$$
% %
% Likewise, $\partial\vol(P^\circ(\mathbf c))/\partial c_i$ is a homogeneous function of degree $d-1$. Thus
\end{proof}

%Assuming that $\tilde\alpha$ can be normalized to $\sum_i\alpha_i=1$, we now see that $\tilde M$ too can be normalized to $\sum_{i,j} M_{ij}=1$. 
%The claim for the normalized quantities follows.

%\begin{corollary}
We immediately see that 
$\sum_{i,j}\tilde M_{ij}>0$ and that we can normalize to $\sum_{i,j}M_{ij} = 1$. For the normalized quantities then indeed holds \eqref{eq:Wachspress_is_Izmestiev_row_sum}:

\begin{corollary}
\label{res:Wachspress_is_Izmestiev_row_sum}
$\sum_{i} M_{ij}=\alpha_j$ for all $j\in\{1,...,n\}$.
\end{corollary}

Lastly, the following properties of the Wachspress coordinates and the~Izmestiev matrix will be relevant and can be inferred from the above.

\begin{samepage}
\begin{remark}\quad
\label{rem:Wachspress_properties}
\begin{myenumerate}
    \item The Wachspress coordinates of the origin and the Izmestiev matrix depend continuously on the translation of $P$, and their normalized variants can be continuously extended to $0\in \partial P$. If the origin lies in the relative interior of a face $\sigma\in\F(P)$, then $\alpha_i>0$ if and only if $p_i\in \sigma$. In particular, if~$0\in$ $\Int(P)$, then $\alpha\in\Int\Delta_n$. % (see \cite{warren1996barycentric,warren2007barycentric,ju2005geometric}).
    \item The Wachspress coordinates of the origin and the Izmestiev matrix are~invariant under linear transformation of $P$. 
    This can be inferred from \eqref{eq:geometric_definition} via an elementary computation, as was done for the Izmestiev matrix in \cite[Proposition 4.6.]{winter2021capturing}.
    
    % are an (almost) local geometric invariant, that is, $\alpha_i$ only depends on a local neighborhood of $p_i$, as well as its distance~from $x$.
\end{myenumerate}
\end{remark}
\end{samepage}

\iffalse % historical notes

\subsection{Historical notes}
\label{sec:appendix_Izmestiev_Wachspress_historical_notes}

Wachspress coordinates were initially defined by Wachspress for polygons \cite{wachspress1975rational}, later generalized to general polytopes by Warren et al.~\cite{warren1996barycentric,warren2007barycentric}. The elegant construction in \eqref{eq:geometric_definition} is from \cite{ju2005geometric}.

The Izmestiev matrix $M\in\RR^{n\x n}$ in dimension three was first constructed and~investigated by Lovász \cite{lovasz2001steinitz} and subsequently generalized to higher dimensions by Izmestiev \cite{izmestiev2010colin}. Another concise proof of the spectral properties of the Izmestiev matrix can be found in the appendix of \cite{narayanan2021spectral}.
In other contexts a matrix $M$ defined variationally as in \eqref{eq:variational_definition} has also been called \emph{Alexandrov matrix} or simply \emph{formal Hessian}.

\fi

%\subsection{Proof of main result}
\subsection{Proof of \cref{res:expansion_main_result}}
\label{sec:proof_of_main_result}

% \begin{theorem}\label{res:expansion_main_result}
% Let $P\subset\RR^d$ be a polytope with edge-graph $G_P=(V,E)$ and $q\: V\to\RR^e$ some embedding thereof. If edges in $q$ are not longer than in $P$, and if $\alpha\in\Delta_n$~is a set of Wachspress coordinates for some interior point $x\in\Int(P)$, then
% %
% $$\|P\|_\alpha\ge \|q\|_\alpha.$$

% Equality holds if and only if $q$ is an affine transformation of the skeleton of $P$, all edges of which are of the same length as in $P$.
% \end{theorem}

Recall the main theorem.

\begin{theoremX}{\ref{res:expansion_main_result}}
Let $P\subset\RR^d$ be a polytope with edge-graph $G_P=(V,E)$ and let~$\alpha\in\Delta_n$ be the Wachspress coordinates of some interior point $x\in\Int(P)$.
If $q\: V$ $\to\RR^e$ is some embedding of $G_P$ whose edges are at most as long as in $P$, then
$$\|P\|_\alpha\ge \|q\|_\alpha,$$
with equality if and only if $q$ is an affine transformation of the skeleton $\skel(P)$,\nlspace all~ed\-ges of which are of the same length as in $P$.
\end{theoremX}

The proof presented below is completely elementary, using little more than~linear algebra.
In \cref{sec:semi_definite_proof} the reader can find a second shorter proof based on~the~duality theory of semi-definite programming.

\begin{proof}
At the core of this proof is rewriting the $\alpha$-expansions $\|P\|_\alpha$ and $\|q\|_\alpha$ as a sum~of terms, each of which is non-increasing when transitioning from $P$ to $q$:
%
% \begin{align}\label{eq:3_term_decomposition}
% \|P\|_\alpha^2 \,
%     &= \sum_{ij\in E} M_{ij}\|p_i-p_j\|^2 - \Big\|\sum_i\alpha_i p_i\Big\|^2 + \tr(MX_PX_P\T)
%   \\&\ge \sum_{ij\in E} M_{ij}\|q_i-q_j\|^2 \,- \Big\|\sum_i\alpha_i q_i\Big\|^2 + \tr(MX_qX_q\T) =\, \|q\|_\alpha^2. \notag
% \end{align}
\begin{align}\label{eq:3_term_decomposition}
\|P\|_\alpha^2 \,
    &= \sum_{ij\in E} M_{ij}\|p_i-p_j\|^2 - \Big\|\sum_i\alpha_i p_i\Big\|^2 + \tr(MX_PX_P\T)
  \\[-1.5ex] &\notag
    \qquad\qquad\quad\; \rotatebox{90}{$\le$} 
    \qquad\qquad\qquad\;\;\, \rotatebox{90}{$\le$} 
    \qquad\qquad\qquad \rotatebox{90}{$\le$} 
  \\[-1.5ex]
  &\;\,\phantom{\ge} \sum_{ij\in E} M_{ij}\|q_i-q_j\|^2 \,- \Big\|\sum_i\alpha_i q_i\Big\|^2 + \tr(MX_qX_q\T) =\, \|q\|_\alpha^2. \notag
\end{align}
Of course, neither the decomposition nor the monotonicity of the terms is obvious; yet their proofs use little more than linear algebra.
We elaborate on this now.

For the setup, 
%\begin{proof}[Proof of \cref{res:expansion_main_result}]
we recall that the $\alpha$-expansion is a translation invariant measure of size. 
We can therefore translate $P$ and $q$ to suit our needs:
\begin{myenumerate}
    %\item choose a point $x\in \Int(P)$ and let $\alpha\in\Int(\Delta_n)$ be the associated Wachspress coordinates.
    \item translate $P$ so that $x=0$, that is, $\sum_i \alpha_i p_i = 0$.
    \item since then $0\in\Int(P)$, \cref{res:Izmestiev} ensures the existence of the Izmestiev matrix $M\in\RR^{n\x n}$. % and its normalized counterpart $\bar M:= M/\vol(P^\circ)$. 
    \item
    Let $\theta_1>\theta_2>\cdots>\theta_m$ be the eigenvalues of $M$, where $\theta_1>0$ and $\theta_2=0$. 
    By~\cref{res:Izmestiev_observations} \itm4 there exists a unique $\theta_1$-eigenvector $z\in\Int(\Delta_n)$.
    \item translate $q$ so that $\sum_i z_i q_i=0$.
\end{myenumerate}

We are ready to derive the decompositions shown in \eqref{eq:3_term_decomposition}: the following equality can be verified straightforwardly by rewriting the square norms as inner products:
\begin{align*}
\tfrac12 \sum_{i,j} M_{ij} \|p_i-p_j\|^2
   &= \sum_i \Big(\sum_j M_{ij}\Big) \|p_i\|^2 - \sum_{i,j} M_{ij}\<p_i,p_j\>,
\end{align*}
% \begin{align*}
% \tfrac12 \sum_{i,j} \bar M_{ij} \|p_i-p_j\|^2
%   %&= \sum_{i,j} \bar M_{ij} \big(\|w_i\|^2 - 2\<w_i,w_j\> + \|w_j\|^2\big)
%   &= \sum_i \underbrace{\Big(\sum_j \bar M_{ij}\Big)}_{\alpha_i}\! \|p_i\|^2 - \underbrace{\sum_{i,j} \bar M_{ij}\<p_i,p_j\>}_{\tr(\bar M X_P X_P\T)},
% \end{align*}
%
We continue rewriting each of the three terms:
\begin{itemize}
    \item on the left: $M_{ij}\|p_i-p_j\|^2$ is only non-zero for $ij\in E$ (using \cref{res:Izmestiev}~\itm2). The sum can therefore be rewritten to iterate over the edges of $G_P$ (where we consider $ij,ji\in E$ the same and so we can drop the factor $\nicefrac12$)
    \item in the middle: the row sums of the Izmestiev matrix are exactly the Wachspress coordinates of the origin, that is, $\smash{\sum_{j} M_{ij}=\alpha_i}$.
    \item on the right: recall the matrix $X_P\in\RR^{d\x n}$ whose rows are the vertex~coordinates of $P$.
    The corresponding Gram matrix $X_P X_P\T$ has entries $(X_PX_P\T)_{ij}=\<p_i,p_j\>$.
\end{itemize}
By this we reach the following equivalent identity:
\begin{align*}
\sum_{ij\in E} M_{ij} \|p_i-p_j\|^2
   &= \sum_i \alpha_i \|p_i\|^2 - \sum_{i,j} M_{ij}(X_P X_P\T)_{ij},
\end{align*}
We continue rewriting the terms on the right side of the equation:
\begin{itemize}
    \item in the middle: the following transformation was previously used in the~simplex case (\cref{res:special_case_simplices}) and can be verified by straightforward expansion~of the squared norms: $$\sum_i\alpha_i \|p_i\|^2 = \tfrac12 \sum_{i,j}\alpha_i\alpha_j \|p_i-p_j\|^2 + \Big\|\sum_i \alpha_i p_i\Big\|^2.$$
    Note that the middle term is just the $\alpha$-expansion $\|P\|_\alpha^2$.
    \item on the right: the sum iterates over entry-wise products of the two matrices $M$ and $X_PX_P\T$, which can be rewritten as $\tr(M X_P X_P\T)$.
\end{itemize}
Thus, we arrive at
\begin{align*}
\sum_{ij\in E} M_{ij} \|p_i-p_j\|^2
   &= \|P\|_\alpha^2 + \Big\|\sum_i \alpha_i p_i\Big\|^2\! - \tr(M X_PX_P\T).
\end{align*}
This clearly rearranges to the first line of \eqref{eq:3_term_decomposition}.
An analogous sequence of transformations works for $q$ (we replace $p_i$ by $q_i$ and $X_P$ by $X_q$, but we keep the Izmestiev matrix of $P$).
This yields the second line of \eqref{eq:3_term_decomposition}.
It remains to verify the term-wise inequalities.

For the first term we have
$$
\sum_{ij\in E} M_{ij} \|p_i-p_j\|^2 \ge \sum_{ij\in E} M_{ij} \|q_i-q_j\|^2
$$
by term-wise comparison: we use that the sum is only over edges, that $M_{ij}>0$ for $ij\in E$ (by \cref{res:Izmestiev} \itm1), and that edges in $q$ are not longer than in $P$.

Next, by the wisely chosen translation in setup \itm1 we have $\sum_i \alpha_i p_i=0$, thus
$$ -\Big\|\sum_i \alpha_i p_i\Big\|^2\! = 0 \ge -\Big\|\sum_i \alpha_i q_i\Big\|^2.$$

The final term requires the most elaboration. By \cref{res:Izmestiev} \itm4 the Izmestiev matrix satisfies $MX_P=0$.
So it suffices to show that $\tr(M X_qX_q\T)$ is non-positive, as then already follows
\begin{equation}
\label{eq:trace_neg}
\tr(M X_P X_P\T) = 0 \overset?\ge \tr(M X_q X_q\T).    
\end{equation}
To prove $\smash{\tr(M X_q X_q\T)\le 0}$ consider the decomposition $\smash{X_q=X_q^1 + \cdots }$ $\smash{+\, X_q^m}$ where $MX_q^k = \theta_k X_q^k$ (since $M$ is symmetric, its eigenspaces are orthogonal and $X_q^k$ is the column-wise orthogonal projecting of $X_q$ onto the $\theta_k$-eigenspace). 
We compute
\begin{align*}
    \tr(M X_q X_q\T)
      &= \sum_{k,\ell} \tr(M X_q^k (X_q^\ell)\T)
    \\&= \sum_{k,\ell} \theta_k \tr(X_q^k (X_q^\ell)\T) && |\; \text{by $MX_q^k = \theta_k X_q^k$}
    \\&= \sum_{k,\ell} \theta_k \tr((X_q^\ell)\T X_q^k) && |\; \text{by $\tr(AB)=\tr(BA)$}
    \\&= \sum_{k} \theta_k \tr((X_q^k)\T X_q^k).  && |\; \text{since $(X_q^\ell)\T X_q^k = 0$ when $k\not=\ell$}
    % \\&= \theta_1\tr((X_q^1)\T X_q^1) + \sum_{k\ge 2} \theta_k \tr((X_q^k)\T X_q^k)
    % \\&\le \theta_1\tr((X_q^1)\T X_q^1) && \hspace{-6em}|\; \text{since $\theta_k\le 0$ for $k\ge 2$ and $\tr((X_q^k)\T X_q^k)\ge 0$}
\end{align*}
Again, we have been wise in our choice of translation of $q$ in setup \itm4: $\sum_i z_i q_i=0$ can be written as $z\T X_q=0$.
Since $z$ spans the $\theta_1$-eigenspace, the columns~of~$X_q$~are therefore orthogonal to the $\theta_1$-eigenspace, hence $X_q^1=0$. We conclude
\begin{align}\label{eq:tr_final_step}
    \tr(M X_q X_q\T)
      &= \sum_{k\ge 2} \theta_k \tr((X_q^k)\T X_q^k) \le 0,
\end{align}
where the final inequality follows from two observations: first, the Izmestiev matrix $M$ has a unique positive eigenvalue $\theta_1$, thus $\theta_k\le 0$ for all $k\ge 2$ (\cref{res:Izmestiev}~\itm5); second $(X_q^k)\T X_q^k$ is a Gram matrix, hence is positive semi-definite and has a non-negative trace.

This finalizes the term-wise comparison and established the inequality \eqref{eq:3_term_decomposition}. It remains to discuss the equality case.
By now we see that the equality $\|P\|_\alpha=\|q\|_\alpha$ is equivalent to term-wise equality in \eqref{eq:3_term_decomposition}; and so we proceed term-wise.

To enforce equality in the first term
$$
\sum_{ij\in E} M_{ij} \|p_i-p_j\|^2 \overset!=  \sum_{ij\in E} M_{ij} \|q_i-q_j\|^2
$$
we recall again that $M_{ij}>0$ whenever $ij\in E$ by \cref{res:Izmestiev} \itm1.
Thus, we require equality $\|p_i-p_j\|=\|q_i-q_j\|$ for all edges $ij\in E$. And so edges in $q$ must be of~the same length as in $P$.

We skip the second term for now and enforce equality in the last term:
$$0=\tr(M X_P X_P\T) \overset!=\tr(M X_q X_q\T) \overset{\eqref{eq:tr_final_step}} = \sum_{k\ge 2} \theta_k \tr((X_q^k)\T X_q^k).$$
Since $\theta_k<0$ for all $k\ge 3$ (\cf\ \cref{res:Izmestiev_observations} \itm3), for the sum on the right~to~vanish we necessarily have
$$\tr((X_q^k)\T X_ q^k)=0\;\text{ for all $k\ge 3$}\quad\implies\quad X_q^k=0\;\text{ for all $k\ge 3$}.$$
%
%We already know that $\smash{X_q^1=0}$. 
Since we also already know $X_q^1=0$, we are left with $\smash{X_q=X_q^2}$, that is, the columns of $X_q$ are in the $\theta_2$-eigenspace (aka.\ the kernel) of $M$. 
In particular,\nlspace \mbox{$\Span X_q\subseteq\ker M$} $=\Span X_P$, where the last equality follows by \cref{res:Izmestiev_observations} \itm2.
It is well-known that if two matrices satisfy $\Span X_q\subseteq\Span X_P$, then the rows of $X_q$ are linear transformations of the rows of $X_P$, that is, $T X_P\T=X_q\T$ for some linear map $T\:\RR^d\to\RR^e$, or equivalently, $q_i =  T p_i$ for all $i\in V$ (see \cref{res:linear_algebra} in the appendix for a short reminder of the proof).
Therefore, $q$ (considered with its original translation prior to the setup) must have been an \emph{affine} transformation of $\skel(P)$.
%Since we translated $P$ and $q$ to suit our means, this linear relation can be a general affine relation with the original translations.

Lastly we note that equality in the middle term of \eqref{eq:3_term_decomposition} yields no new constraints.
In fact, by $\Span X_q\subseteq \ker M$ we have $MX_q=0$ and
$$\sum_i\alpha_iq_i = \sum_i\Big(\sum_j M_{ij}\Big) q_i=\sum_j\Big(\sum_i M_{ij} q_i\Big)=0=\sum_i\alpha_i p_i.$$
%
%$$\sum_i M_{ij} q_i = 0,\;\;\text{for all $j\in V$}\quad\implies\quad \sum_i\alpha_iq_i = \sum_i \Big(\sum_j M_{ij}\Big) q_i = 0.$$
%
Thus, identity in the middle term follows already from identity in the last term.

For the other direction of the identity case assume that $q$ is an affine transformation of $\skel(P)$ with the same edge lengths.
%Assume a translation of $P$ as in \itm1, \ie\ $\sum_i\alpha_i p_i=0$, and let~$M$~be~the corresponding Izmestiev matrix. 
%Recall $MX_P=0$ by \cref{res:Izmestiev_observations} \itm2.
Instead of setup \itm4 assume a
%Choose furthermore a 
translation of $q$ for which it is a \emph{linear} transformation of $\skel(P)$, \ie\ $\smash{X_q\T = T X_P\T}$ for some linear map $T\:\RR^d\to\RR^e$.
Hence $\smash{\sum_i \alpha_i p_i = \sum_i \alpha_i q_i=0}$ and $M X_P=MX_q=0$, and \eqref{eq:3_term_decomposition} reduces to
$$\|P\|_\alpha^2 = \sum_{ij\in E} M_{ij}\|p_i-p_j\|^2 - 0 + 0 = \sum_{ij\in E} M_{ij}\|q_i-q_j\|^2 - 0 + 0 = \|q\|_\alpha^2.$$
%
%which finishes the proof.
%\end{proof}
\end{proof}%\hfill$\square$

%\subsection{Immediate consequences}

As an immediate consequence we have the following:

\begin{corollary}\label{res:polytope_by_lenghts_and_alpha}
A polytope is uniquely determined (up to affine equivalence) by its edge-graph, its edge lengths and the Wachspress coordinates of some interior point.
\begin{proof}
Given polytopes $P_1$ and $P_2$ with the same edge-graph and edge lengths as well as points $x_i\in\Int(P_i)$ with the same Wachspress coordinates $\alpha\in\Delta_n$.\nls  By~\cref{res:expansion_main_result} we have $\|P_1\|_\alpha\ge\|P_2\|_\alpha\ge\|P_1\|_\alpha$, thus $\|P_1\|_\alpha=\|P_2\|_\alpha$.
Then $P_1$~and~$P_2$ are~affinely equi\-valent by the equality case of \cref{res:expansion_main_result}.
\end{proof}
\end{corollary}

Remarkably, this reconstruction works across all combinatorial types~and dimen\-sions.
That the reconstruction is only up to affine equivalence rather than~iso\-metry is due to examples such as rhombi and the zonotope in \cref{fig:length_preserving_linear_trafo}. In general, such flexibility via an affine transformation happens exactly if ``the edge directions lie on a conic at infinity''
(see \cite[Proposition 4.2]{connelly2005generic} or 
\cite[Proposition 1.4]{connelly2018affine}).
%\cite[Proposition 1.4]{connelly2018affine}

\begin{figure}[h!]
    \centering
    \includegraphics[width=0.53\textwidth]{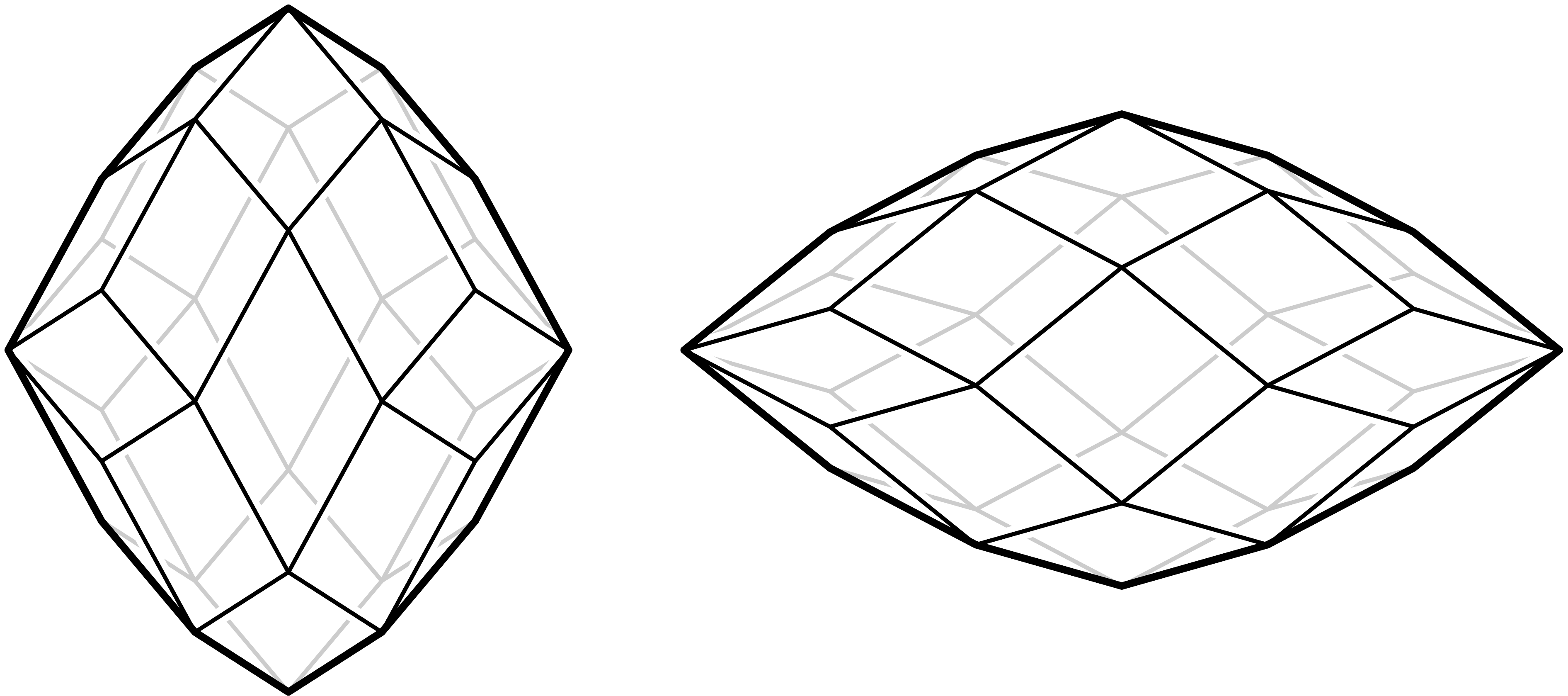}
    \caption{\mbox{Two affinely equivalent (but not isometric) polytopes~with~the} same edge length. The edge directions trace a circle on the ``plane~at~infinity''.}
    \label{fig:length_preserving_linear_trafo}
\end{figure}

Lastly, the reconstruction permitted by \cref{res:polytope_by_lenghts_and_alpha} is feasible in practice. This follows from a reformulation of \cref{res:expansion_main_result} as a semi-definite program, which can be solved in polynomial time. This is elaborated on in the alternative proof given in \cref{sec:semi_definite_proof}.

\section{Rigidity, tensegrity and reconstruction}
\label{sec:expansion_rigidity_tensegrity}

%After the elaborate proof in the last section, let us 
Our reason for pursuing \cref{res:expansion_main_result} in \cref{sec:expansion_version} was to transfer the proof of the simplex case (\cref{res:special_case_simplices}) to general polytopes with the eventual goal of verifying the main conjecture and its corresponding \enquote{unique reconstruction version}:

% After the quite elaborate proof of \cref{res:expansion_main_result}, let us recall why we proved this theorem in the first place: to transfer the proof of the simplex case (\cref{res:special_case_simplices}) to general~polytopes.
% In this section we discuss in how far this was successful, that is, to what extend we can now answer

% In this section we discuss the various consequences of \cref{res:expansion_main_result} that do not directly involve the Wachspress coordinates but are concerned purely with distance data in the polytopes.
% %
% At the horizon the goal is to prove the following conjecture:

\begin{conjectureX}{\ref{conj:main}}
%\begin{conjecture}\label{conj:main}
Given polytopes $P\subset\RR^d$ and $Q\subset\RR^e$ with common edge-graph.\nlspace If
\begin{myenumerate}
    \item $0\in \Int(Q)$,
    \item edges in $Q$ are at most as long as in $P$, and
    \item vertex-origin distances in $Q$ are at least as large as in $P$,
\end{myenumerate}
then $P\simeq Q$.
\end{conjectureX}

%The corresponding \enquote{unique reconstruction} conjecture is

\begin{conjectureX}{\ref{conj:main_rigid_intro}}
A polytope $P$ with $0\in\Int(P)$ is uniquely determined (up to~isometry) by its edge-graph, edge lengths and vertex-origin distances.
\end{conjectureX}

In contrast to our formulation of \cref{res:expansion_main_result}, both of the above conjectures~are \emph{false} when stated for a general graph embedding $q\:V(G_P)\to\RR^e$ instead of $Q$,\nls even if we require $0\in\Int\conv(q)$.
%
%We emphasize that both conjectures are \emph{false} when replacing $Q$ by \mbox{some~general} graph embedding $q\: V(G_P)\to\RR^e$ with $0\in\Int\conv(q)$. 
The following counterexample was provided by Joseph Doolittle \cite{410625}:

% In this section we explain why the work of the last sections is not quite enough to prove \cref{conj:main}, and what we can actually prove.

% This conjecture is specifically stated for polytopes.
% In fact, $Q$ cannot be replaced by a general graph embedding $q\: V(G_P)\to\RR^e$ as demonstrated with the following counterexample by Joseph Doolittle:

%The following counterexample was constructed by Joseph Doolittle (see also \cite{mathoverflow}) which shows that \cref{conj:main} does not hold if $Q$ is replaced by a graph embedding $q\:V(G_P)\to\RR^e$.

\begin{example}\label{ex:cube}
The 3-cube $P:=[-1,1]^3\subset\RR^3$ is inscribed in a sphere of~radius $\smash{\sqrt 3}$.
\Cref{fig:counterexample_cube} shows an inscribed embedding $q\:V(G_P)\to\RR^2$ with the same circumradius and edge lengths, collapsing $G_P$ onto a path.
In the circumcircle each edge spans~an arc of length
$$2\sin^{-1}(1/\sqrt 3)\approx 70.5287^\circ > 60^\circ.$$
The three edges therefore suffice to reach more than half around the circle.
In other words, the convex hull of $q$ contains the circumcenter in its interior.

A full-dimensional counterexample in $\RR^3$ can be obtained by interpreting $q$~as embedded in $\RR^2\times\{0\}$ follows by a slight perturbation.
%Embedding $q$ into $\RR^3$ with slight perturbation yields a full-dimensional counter\-example there as well.
\end{example}

\begin{figure}[h!]
    \centering
    \includegraphics[width=0.5\textwidth]{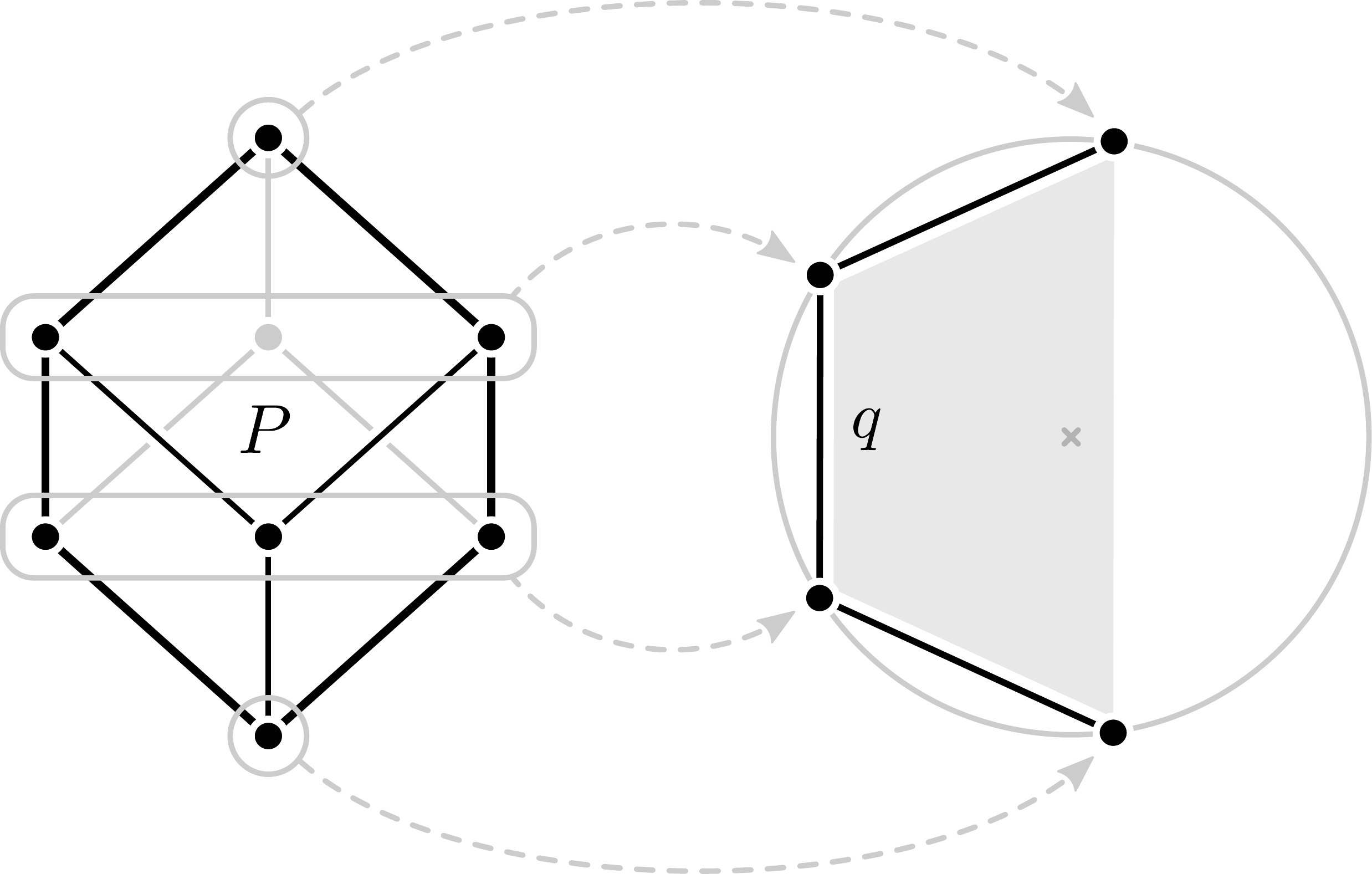}
    \caption{\mbox{A 2-dimensional embedding of the edge-graph~of~the~3-di\-men}\-\mbox{sio\-nal cube with the same circumradius and edge lengths as the unit cube,} also containing the origin in its convex hull.}
    \label{fig:counterexample_cube}
\end{figure}

Potential fixes to the \enquote{graph embedding versions} of \cref{conj:main_rigid_intro} and \cref{conj:main} are discussed in \cref{sec:fixing_q_version_conjecture}.

% \begin{conjecture}\label{conj:josephs_conjecture}
% %\begin{conjecture}\label{conj:main}
% Given a polytope $P\subset\RR^d$ and a graph embedding $q\:V(G_P)\to\RR^e$ so that
% %
% \begin{myenumerate}
%     \item $0\in \Int\conv(q)$,
%     \item edges in $q$ are at most as long as in $P$,
%     \item vertex-origin distances in $q$ are at least as large as in $P$, and
%     \item for each $i\in V(G_P)$ the cone 
%     %
%     $$p_i + \cone\{p_j-p_i\mid ij\in E(G_P)\}$$
%     %
%     contains the origin,
% \end{myenumerate}
% %
% then $\skel(P)\simeq q$.
% \end{conjecture}

%We were not able to prove this conjecture in full generality.
% Note that in the setting of \cref{conj:main} we cannot replace $Q$ by a general~graph embedding $q\:V(G_P) \to \RR^e$ (see \cref{ex:cube} for a counterexample), though below we shall discuss special cases in which this generalization is possible.

% The following weaker \enquote{unique reconstruction statement} would be a consequence of \cref{conj:main}:

% \begin{conjecture}\label{conj:main_rigid}
% A polytope $P$ with $0\in\Int(P)$ is uniquely determined (up to~isometry) by its edge-graph, edge lengths and vertex-origin distances.
% \end{conjecture}

%Despite the strong geometric data this is still remarkable as it determines the polytopes across all dimensions and combinatorial types.

While the general \cref{conj:main} will stay open, we are confident that our~methods point the right way and highlight the essential difficulties. 
We overcome~them in three relevant special cases, for some of which we actually can replace~$Q$ with a graph~embedding $q\:V(G_P)\to\RR^e$. Those are
%While the general \cref{conj:main} will stay open for now, in this section we verify the following three special cases, for some of which we actually can replace $Q$ with a graph~embedding $q\:V(G_P)\to\RR^e$:
%
\begin{myenumerate}
    \item $P$ and $q$ are centrally symmetric (\cref{res:centrally_symmetric}). % (or have otherwise distinguished \enquote{center points}).
    \item $q$ is a sufficiently small perturbation of $\skel(P)$ (\cref{thm:main_local}).
    \item $P$ and $Q$ are combinatorially equivalent (\cref{thm:main_comb_eq}).
\end{myenumerate}

\subsection{The remaining difficulty}

On trying to generalize the proof of \cref{res:special_case_simplices} beyond simplices using \cref{res:expansion_main_result} we face the following difficulty:
%
%Note however the following subtlety:
\cref{res:expansion_main_result} requires the $\alpha\in\Delta_n$ to be Wachspress coordinates of an interior point $x\in\Int(P)$, whereas in the proof of \cref{res:special_case_simplices} we use that $\alpha$ is a set of barycentric coordinates of $0\in$ $\Int(Q)$. 
While we have some freedom in choosing $x\in\Int(P)$, and thereby $\alpha\in\Delta_n$, it is not clear that any such choice yields barycentric coordinates for $0\in\Int(Q)$.
In fact, this is the only obstacle preventing us from proving \cref{conj:main} right away.
For convenience we introduce the following map:

\begin{definition}\label{def:Wachspress_map}
Given polytopes $P\subset\RR^d$ and $Q\subset\RR^e$, the \emph{Wachspress map}~$\phi\: P\to Q$ is defined as follows: for $x\in P$ with Wachspress coordinates $\alpha(x)\in\Delta_n$ set
$$\phi(x):=\sum_i \alpha_i(x) q_i.$$
%
%This map is called the \!\emph{Wachspress map} from $P$ to $Q$.
% Given a polytope $P\subset\RR^d$ and a graph embedding $q\:V(G_P)\to\RR^e$, the \emph{Wachspress map} $\phi\: P\to\conv(q)$ is defined as follows: for $x\in P$ with Wachspress coordinates $\alpha\in\Delta_n$ set
% %
% $$\phi(x):=\sum_i \alpha_i q_i.$$
\end{definition}

In cases where we are working with a graph embedding $q\:V(G_P)\to \RR^e$ instead of $Q$ we have an analogous map $\phi\: P\to\conv(q)$.

Our previous discussion amounts to checking whether the origin is in the image of $\Int(P)$ \wrt\ $\phi$.
While this could be reasonably true if $\dim P \ge \dim Q$, it~is~certainly too much to hope for if $\dim P<\dim Q$: the image of $\phi(\Int(P))\subset Q$ is of a smaller dimension than $Q$ and easily ``misses'' the origin.
Fortunately, we can ask for less, which we now formalize in \itm1 of the following lemma:

\begin{lemma}\label{lem:main}
Let $P\subset\RR^d$ be a polytope and $q\:V(G_P)\to\RR^e$ some embedding. If
\begin{myenumerate}
    %\item $0\in \phi (\Int(P))$.
    \item there exists an $x\in\Int(P)$ with $\|\phi(x)\|\le\|x\|$ (\eg\ $\phi(x)=0$),
    \item edges in $q$ are at most as long as in $P$, and
    \item vertex-origin distances in $q$ are at least as large as in $P$,
\end{myenumerate}
then $q\simeq\skel(P)$ (via an orthogonal transformation). 
%P\simeq q$ (in the sense that the skeleton of $P$ is isometric to $q$).
%

\begin{proof}
Choose $x\in \Int P$ with $\|\phi(x)\|\le \|x\|$, and note that its Wachspress coordi\-nates $\alpha\in\Int\Delta_n$ are strictly positive.
In the remainder we follow closely the proof of \cref{res:special_case_simplices}: consider the system of (in)equalities:
\begin{align*}
    \sum_i \alpha_i \|p_i\|^2
      &= \Big\|\sum_i \alpha_i p_i\Big\|^2 \!+
         \tfrac12\sum_{i,j} \alpha_i\alpha_j \|p_i-p_j\|^2
      = \;\; \|x\|^2 \;\;\;\,+ \|P\|_\alpha^2
    \\[-1.5ex]
    \rotatebox{90}{$\ge$} 
    \qquad
    & 
    \qquad \qquad \qquad \qquad \qquad \qquad \qquad \qquad \qquad \,\;\;
    \rotatebox{90}{$\le$}
    \qquad\;\;\;\;\;\,
    \rotatebox{90}{$\le$}
    \\[-1.5ex]
    \sum_i \alpha_i \|q_i\|^2
      &= \Big\|\sum_i \alpha_i q_i\Big\|^2 \!+
         \tfrac12\sum_{i,j} \alpha_i\alpha_j \|q_i-q_j\|^2\,
      = \|\phi(x)\|^2 \,+\, \|q\|_\alpha^2,
\end{align*}
where the two rows hold by simple computation and the vertical inequalities follow (from left to right) by \itm3, \itm1, and \itm2 + \cref{res:expansion_main_result} respectively.
It follows~that all inequali\-ties are actually equalities.
In particular, since $\alpha_i>0$ we find both~$\|p_i\|$ $=$ $\|q_i\|$ for all $i\in V$ and $\|p_i-p_j\|=\|q_i-q_j\|$ for all $i,j\in V(G_P)$, establishing that $q$ and $\skel(P)$~are indeed isometric via an orthogonal transformation.
%
%, in particular, $\|P\|_\alpha=\|q\|_\alpha$, which implies $q\simeq \skel(P)$ by the equality case of \cref{res:expansion_main_result}.
\end{proof}
\end{lemma}

The only way for \cref{lem:main} \itm1 to fail is if $\|\phi(x)\|>\|x\|$ for all $x\in\Int(P)$.\nls
By \itm2 and \itm3 we have $\|\phi(x)\|=\|x\|$ whenever $x$ is a vertex or in an edge of~$P$, which makes it plausible that \itm1 should not fail, yet it seems hard to exclude~in~general.

In each of the three special cases of \cref{conj:main} discussed below we actually managed to verify $0\in\phi(\Int(P))$ in order to apply \cref{lem:main}.

\subsection{Central symmetry}
\label{sec:central_symmetry}

Let $P\subset\RR^d$ be \emph{centrally~symmetric} (more precisely, \emph{origin symmetric}), that is, $P=-P$.
%
%For an arbitrary embedding $q\:V(G_P)\to\RR^e$ of the edge-graph \enquote{centrally symmetric} will mean $p_i=-p_j\implies q_i=-q_j$.
%
This induces an involution \mbox{$\iota\: V(G_P)\to V(G_P)$} with $p_{\iota(i)}=-p_i$ for all $i\in V(G_P)$. We say that~an~embedding $q\: V(G_P)\to\RR^d$ of the edge-graph is centrally symmetric if $q_{\iota(i)}=$ $-q_i$ for all $i\in V(G_P)$.

\begin{theorem}[centrally symmetric version]\label{res:centrally_symmetric}
Given a centrally symmetric polytope $P\subset\RR^d$ and a centrally symmetric graph embedding $q\: V(G_P)\to\RR^e$, so that
\begin{myenumerate}
    \item edges in $q$ are at most as long as in $P$, and
    \item vertex-origin distances in $q$ are at least as large as in $P$,
\end{myenumerate}
then $q\simeq \skel(P)$.
\begin{proof}
Since $P$ is centrally symmetric, we have $0\in\relint(P)$ and we can find~Wachspress coordinates $\alpha\in\Int(\Delta_n)$ of the origin in $P$.
Since the Wachspress coordinates are invariant under a linear transformation (as noted in \cref{rem:Wachspress_properties} \itm2), it holds $\alpha_i = \alpha_{\iota(i)}$. 
For the Wachspress map $\phi$ follow
%
% $$
% \phi(0) + \phi(0) 
% = \sum_{i\in V} \alpha_i q_i + \sum_{i\in V} \alpha_{\iota(i)} q_{\iota(i)}
% = \sum_{i\in V} \alpha_i q_i - \sum_{i\in V} \alpha_i q_i = 0,
% $$
$$
\phi(0) 
= \tfrac12 \sum_{i\in V} \alpha_i q_i + \tfrac12 \sum_{i\in V} \alpha_{\iota(i)} q_{\iota(i)}
= \tfrac12 \sum_{i\in V} \alpha_i q_i - \tfrac12\sum_{i\in V} \alpha_i q_i = 0,
$$
The claim then follows via \cref{lem:main}
\end{proof}
\end{theorem}

It is clear that \cref{res:centrally_symmetric} can be adapted to work with other types of symmetry that uniquely fix the origin, such as irreducible symmetry groups.

\Cref{res:centrally_symmetric} has a natural interpretation in the language of classical rigidity~theory, where it asserts the universal rigidity of a certain tensegrity framework. In this form it was proven up to dimension three by Connelly \cite[Theorem 5.1]{connelly2006stress}. We elaborate further on this in \cref{sec:classical_rigidity}.

It is now tempting to conclude the unique reconstruction version of \Cref{res:centrally_symmetric}, answering \cref{conj:main_rigid_intro} for centrally symmetric polytope.
There is~however~a~subtlety: 
our notion of ``central symmetry'' for the graph embedding $q\:V(G_P)\to\RR^e$ as used in \cref{res:centrally_symmetric} has been relative to $P$, in that it forces $q$ to have the same pairs of antipodal vertices as $P$.
It is however not true that any two centrally symmetric polytopes with the same edge-graph have this relation. 
David E.\ Speyer \cite{antipodalNotUnique} constructed a 12-vertex 4-polytope whose edge-graph has an automorphism that does not preserve antipodality of vertices.

\subsection{Local uniqueness}
\label{sec:local_rigidity}

% Given a polytope $P\subset\RR^d$, the space $\mathcal E_P:=\{q:V(G_P)\to\RR^d\}$ of embeddings of $G_P$ carries a natural topology pulled back from $\RR^{nd}$.
% It therefore makes sense to speak of neighborhoods and convergence of embeddings.

Given a polytope $P\subset\RR^d$ consider the space 
$$\mathcal E_P:=\{q:V(G_P)\to\RR^d\}$$
of $d$-dimensional embeddings of $G_P$.
We then have $\skel(P)\in\mathcal E_P$.
Since $\mathcal E_P\cong~\RR^{n\x d}$ (in some reasonable sense) we can pull back a metric $\mu\:\mathcal E_P\times\mathcal E_P\to\RR_{+}$.\;
%~to~have a notion of $\epsilon$-balls and uniform convergence.

% Given a polytope $P\subset\RR^d$, the embeddings of its edge-graph form a space with a natural topology.

% We now establish that polytope skeleta with central rods are rigid in the usual sense of rigidity theory, or more generally, in the sense of tensegrity.

\begin{theorem}[local version]\label{thm:main_local}
% Given a polytope $P\subset\RR^d$ with $0\in\relint(P)$, there exists a neighborhood $N\subset\mathcal E_P$ of $\skel(P)$, so that for all embeddings $q\in N$ for which
% Given a polytope $P\subset\RR^d$ with $0\in\Int(P)$, there exists a neighborhood $N\subset\mathcal E_P$ of $\skel(P)$ with the following property: if $q\in N$ with
Given a polytope $P\subset\RR^d$ with $0\in\Int(P)$, there~exists an $\eps >0$ with the following property: if $q\:V(G_P)\to\RR^d$ is an embedding~with
\begin{myenumerate}
    \item $q$ is $\eps$-close to $\skel(P)$, \ie\ $\mu(q,\skel(P))<\eps$,
    \item edges in $q$ are at most as long as in $P$, and
    \item vertex-origin distances in $q$ are at least as large as in $P$,
\end{myenumerate}
then $q\simeq \skel(P)$.
\end{theorem}

\Cref{thm:main_local} as well can be naturally interpreted in the language of rigidity~theory (see \cref{sec:classical_rigidity}).
The proof below makes no use of this language. %, but in \cref{sec:classical_rigidity} we further comment on this connection. % and sketch a second approach that proves \emph{prestress stability}.
%
%This result too was proven up to dimension three by Connelly \cite{connelly2006stress}.

% \begin{disable}
% \begin{remark}
% \label{rem:rigidity_interpretation}
% \Cref{thm:main_local} has a natural interpretation in the language of rigidity and tensegrity structures: consider a bar-joint framework where every edge in $\F_1(P)$ is considered as a rigid bar, connected at flexible joints located in the vertices $\F_0(P)$.
% %with joints at $\F_0(P)$ and $x$, a bar for each edge $\F_1(P)$ and additional bars connecting $x$ to $\F_0(P)$.
% \Cref{thm:main_local} then states that this structure is rigid in the sense of classical rigidity theory. % \cite{...}.
% Moreover, we can replace the edge bars by cables \mbox{(which~are~allowed~to~con}\-tract) and the central bars by struts (which are allows to expand), and the structure would still be rigid as a so-called tensegrity structure. 
% % We compare this to the well-known result that skeletons of \emph{simplicial} polytopes (all facets are simplices) is rigid as a bar-joint framework.
% % This is not true for polytopes that are not simplicial, but as we have shown this can be rectified by adding \enquote{central struts}.
% \end{remark}
% \end{disable}

%\msays{For this version one needs to keep the dimension because otherwise there are counterexamples. Does one need to keep the dimension in the cable/strut flipped version?} % --> YES

%Note that we do not assume that $P$ or $q$ are centrally symmetric.
%\msays{Meh...}

In order to prove \cref{thm:main_local} we again show $0\in\phi(\Int(P))$, which requries more work this time: since $0\in\Int(P)$, there exists an $\eps$-neighborhood $B_\eps(0)\subset P$~of the origin. 
The hope is that for a sufficiently small perturbation of the vertices of~$P$~the image of $B_\eps(0)$ under $\phi$ is still a neighborhood of the origin.

This hope is formalized and verified in the following lemma, which we separated from the proof of \cref{thm:main_local} to reuse it in \cref{sec:combinatorial_equivalence}.
Its proof is standard and is included in \cref{sec:topological_argument}:

% This hope is justified and probably not surprising to a reader well-versed in topology.
% It is formalized in the following lemma, the proof of which we outsourced to \cref{sec:topological_argument}, for once to stay focused, and also since the result is reused in the next section.

\begin{lemmaX}{D.1}%[proven in \cref{sec:topological_argument}]
Let $K\subset\RR^d$ be a compact convex set with $0\in\Int(K)$ and $f\:K\times[0,1]$ $\to\RR^d$ a homotopy with $f(\free,0)=\id_K$. 
If the restriction $f|_{\partial K}\: \partial K\times[0,1]\to\RR^d$ yields a homotopy of $\partial K$ in $\RR^d\nozero$, then $0\in\Int f(K,1)$.
\end{lemmaX}
% \begin{proof}
% see \cref{sec:topological_argument}.
% \end{proof}

In other words: if a \enquote{well-behaved} set $K$ contains the origin in its interior,\nlspace and~it is~deformed so that its boundary never crosses the origin, then the origin stays~in\-side.

\begin{proof}[Proof of \cref{thm:main_local}]
Since $0\in\Int(P)$ there exists a $\delta>0$ with $B_\delta(0)\subset P$.

Fix some compact neighborhood $N\subset\mathcal E_P$ of $\skel(P)$.
Then $N\times P$ is compact in $\mathcal E_P\times \RR^d$ and the map
$$N\times P\to\RR^d,\;\;(q,x) \mapsto \phi_q(x):=\sum_i \alpha_i(x)q_i$$
is uniformly continuous: there exists an $\epsilon>0$ so that whenever $\mu(q,q')+\|x-x'\|< \eps$, we have $\|\phi_q(x)-\phi_{q'}(x')\|<\delta/2$. 
We can assume that $\eps$ is sufficiently small,\nls so~that $B_\eps(\skel(P))\subset N$.
We show that this $\eps$ satisfies the statement of the theorem.

Suppose that $q$ is $\eps$-close to $\skel(P)$, then
$$\mu(\skel(P),q) + \|x-x\|<\eps \;\,\implies\;\, \|x-\phi_q(x)\|=\|\phi_{\skel(P)}(x) - \phi_q(x)\|<\delta/2.$$
The same is true when replacing $q$ by any linear interpolation $q(t):=(1-t)\skel(P)+t q$ with $t\in[0,1]$.
Define the following homotopy:
$$f\: B_\delta(0)\times[0,1]\to\RR^d,\quad (x,t)\mapsto \phi_{q(t)}(x).$$
We have $f(\free,0)=\id$. 
%On the other hand, 
That is, if $x\in\partial B_\delta(0)$ then $\|f(x,0)\|=\delta$, as well as
$$\|f(x,t)\| \ge \|f(x,0)\| - \|f(x,0)-f(x,t)\| \ge \delta - \delta/2=\delta/2,$$
and $f(x,t)\not=0$ for all $t\in[0,1]$.
Since $B_\delta(0)$ is compact and convex, the homotopy $f$ satisfies the conditions of \cref{res:topological_argument} and we conclude $0\in f(B_\delta(0),1)=\phi_q(B_\delta(0))$ $\subseteq\phi_q(\Int(P))$.
Then $q\simeq \skel(P)$ follows via \cref{lem:main}.
\end{proof}

\begin{figure}[h!]
    \centering
    \includegraphics[width=0.7\textwidth]{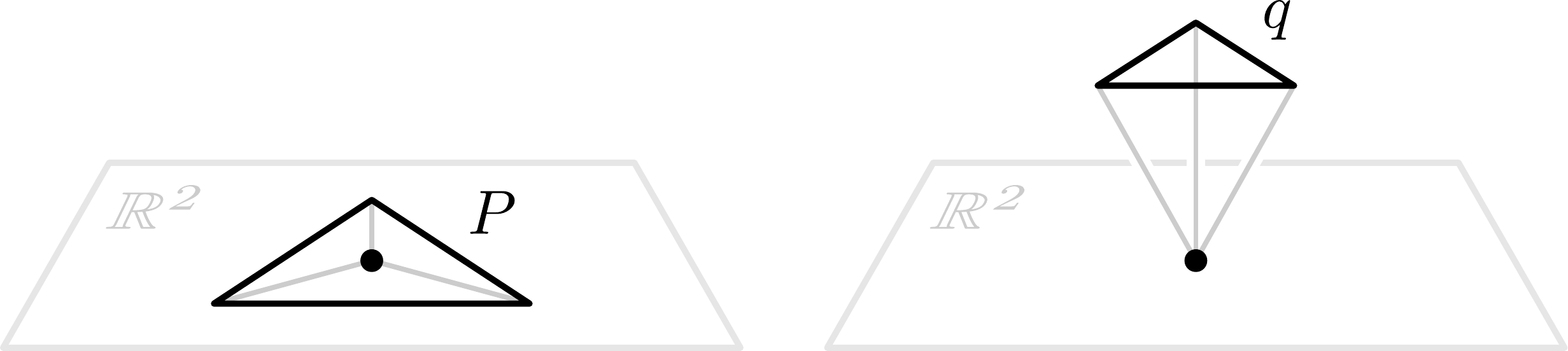}
    \caption{If $q$ in \cref{thm:main_local} is not restricted to $\aff(P)$, the \mbox{vertex-ori}\-gin distances can be increased by moving out of the affine hull.}
    \label{fig:dimension_increase_ex}
\end{figure}

The polytope $P$ in \cref{thm:main_local} is assumed to be full-dimensional.
This is necessary, since allowing $\skel(P)$ to deform beyond its initial affine hull already permits counterexamples such as shown in \cref{fig:dimension_increase_ex}. Even restricting to deformations with $0\in\Int\conv(q)$ is not sufficient, as shown in the next example:

\begin{example}
\label{ex:cube_counterexample_2}
Consider the 3-cube as embedded in $\RR^3\times\{0\}\subset\RR^3\times\RR^2\cong\RR^5$.
Let $p_1,...,p_8\in\RR^3$ $\times$ $\{0\}$ be its vertices, and let $q_1,...,q_8\in\{0\}\times\RR^2$ be the vertices as embedded in \cref{fig:counterexample_cube} (on the right). Since both share the same edge lengths and vertex-origin distances, so does the em\-bedding $tp+sq$ whenever $t^2+s^2=1$.
Observe further that for both $p_i$ and $q_i$~the origin can be written as a convex combination using the same coefficients $\alpha\in\Delta_n$ (use an $\alpha$ with $\alpha_i=\alpha_j$ whenever $p_i=-p_j$). It follows $0\in$ $\Int\conv(tp+sq)$. % with $\alpha_i=\alpha_j$ whenever $p_i=-p_j$.
\end{example}

% {\color{lightgray}
% The dual version of \cref{thm:main_local} (exchanging cables and struts in the rigidity interpretation) remains open and we state it~as a conjecture in \cref{q:thm_local_complementary}.
% Note however that \cref{ex:cube_counterexample_2} shows that here too we need $P$ of full dimension.
% }

%It remains open whether a deformation is possible that expands edges and contract vertex-origin distances (see \cref{q:...}).

% The following questions cannot be answered by the same technique:

% \begin{question}
% If $P$ is not necessarily full-dimensional, is $\skel(P)$ still rigid under~the additional constraint to keep $0\in\relint\conv(q)$?
% \end{question}

% \begin{question} \msays{NO! \cref{fig:flex_pyramid}
% If $0\not\in\relint(P)$, is the skeleton with central rods still rigid as a bar-joint framework?
% \end{question}

\subsection{Combinatorially equivalent polytopes}
\label{sec:combinatorial_equivalence}

%We can say much more if we not only prescribe the edge-graph, but the full face lattice. 
In this section we consider combi\-natorially equivalent polytopes $P,Q\subset\RR^d$ and prove the following:

\begin{theorem}[combinatorial equivalent version]\label{thm:main_comb_eq}
Let $P,Q\subset\RR^d$ be combinatorially equivalent polytopes so that
\begin{myenumerate}
    \item $0\in\Int(Q)$
    \item edges in $Q$ are at most as long as in $P$, and
    \item vertex-origin distances in $Q$ are at least as large as in $P$.
\end{myenumerate}
Then $P\cong Q$.
\end{theorem}

In particular, since the combinatorics of polytopes up to dimension three~is~deter\-mined by the edge-graph, this proves \cref{conj:main} for $d,e\le 3$.

Once again the proof uses \cref{lem:main}.
Since $0\in\Int(Q)$, we can verify \mbox{$0\in\phi(\Int(P))$} by showing that the Wachspress map $\phi\: P\to Q$ is surjective.
This statement is of inde\-pendent interest, since the question whether $\phi$ is \emph{bijective} is a \mbox{well-known} open problem for $d\ge 3$ (see \cref{conj:Wachspress_injective}).
Our proof of surjectivity uses~\cref{res:topological_argument} and the following:

% One way to see $0\in\phi(\Int(P))$ would be to show that the Wachspress map $\phi\:P\to Q$ is a bijection. In fact, this is a long-standing unsolved conjecture (see \cref{q:Wachspress_map_injective}) only known to be true in dimension $d=2$.
% Luckily, for our purpose it suffices to show that $\phi$ is surjective, which we can prove, again with the help of \cref{res:topological_argument}.

% The proof of the result however needs some preparation.
% Recall that we need to find an $x\in\Int(P)$ so that $\phi(x)=0\in Q$, where $\phi\: P\to Q$ is the Wachspress map (\cf\ \cref{def:Wachspress_map}).
% This would be clear if this map would be bijective, thought this is unknown. In fact, the following is a famous open question:

% \begin{question}
% Given combinatorially equivalent polytopes $P,Q\subset\RR^d$, is the Wachspress map $\phi\:P\to Q$ \emph{injective}?
% \end{question}

% This is know  to be true for $d=2$ \cite{floater}, but open for all $d\ge 3$.
% %
% Luckily, we only need $\phi$ to be \emph{surjective}, which we establish in next two lemmas:

\begin{lemma}\label{res:Wachspress_map_sends_boundary}
Given a face $\sigma\in\F(P)$, the Wachspress map $\phi$ sends $\sigma$ onto the~corresponding face $\sigma_Q\in\F(Q)$.
In particular, $\phi$ sends $\partial P$ onto $\partial Q$.
\end{lemma}

\begin{proof}
Given a point $x\in\relint(\sigma)$ with Wachspress coordinates $\alpha\in\Delta_n$, the coeffi\-cient $\alpha_i$ is non-zero if and only if the vertex $p_i$ is in $\sigma$ (\cref{rem:Wachspress_properties} \itm1). 
The claim $\phi(x)\in\sigma_Q$ follows immediately.
\end{proof}

\begin{lemma}\label{res:Wachspress_map_surjective}
The Wachspress map $\phi\:P\to Q$ is surjective.
\begin{proof}
We proceed by induction on the dimension $d$ of $P$.
For $d=1$ the Wachspress map is linear and the claim is trivially true. 
For $d>1$ recall that $\phi$ sends $\partial P$ to~$\partial Q$ (by \cref{res:Wachspress_map_sends_boundary}).
By induction hypothesis, $\phi$ is surjective on each proper face,\nls thus sur\-jective on all of $\partial P$.

%Recall that if $\sigma\in\F(P)$ is a face of $P$, $x\in \relint(\sigma)$ some point, and $\alpha\in\Delta_n$ its~Wachs\-press coordinates in $P$, then $\alpha_i>0$ if and only if the corresponding vertex $p_i$ is in $\sigma$.
%As a consequence, the Wachspress map sends $\sigma$ to the corresponding face $\sigma_Q\in\F(Q)$. In particular, it sends the boundary $\partial P$ to $\partial Q$.

To show surjectivity in the interior, we fix $x\in\Int(Q)$; we show $x\in\im\phi$.
Let $\psi\:Q\to P$ be the Wachspress map in the other direction (which is usually \emph{not} the inverse of $\phi$) and define $\rho:=\phi\circ\psi\:Q\to Q$.
Note that by \cref{res:Wachspress_map_sends_boundary} $\rho$ sends each face of $Q$ to itself and is therefore homotopic to the identity on $Q$ via the following linear homotopy:
$$f\:Q\times[0,1]\to Q,\; (y,t)\mapsto (1-t)x+t\rho(x).$$
Since faces of $Q$ are closed under convex combination, $f(\free,t)$ sends $\partial Q$ to itself for all $t\in[0,1]$.
Thus, $f$ satisfies the assumptions of \cref{res:topological_argument} (with $K$ chosen as $Q$), and therefore $x\in f(Q,1)=\im \rho\subset\im\phi$.
\end{proof}
\end{lemma}

The proof of \cref{thm:main_comb_eq} follows immediately:

\begin{proof}[Proof of \cref{thm:main_comb_eq}]
Since $0\in\Int(Q)$ and the Wachspress map $\phi\:P\to Q$~is~sur\-jective (by \cref{res:Wachspress_map_surjective}), there exists $x\in P$ with $\phi(x)=0$.
Since $\phi(\partial P)=\partial Q$ (by \cref{res:Wachspress_map_sends_boundary}), we must have $x\in\Int(P)$.
$P\simeq Q$ then follows via \cref{lem:main}.
\end{proof}

%We formulate a \enquote{unique reconstruction version} of \cref{thm:main_comb_eq}, analogous to \cref{res:reconstruct_centrally_sym}, though we need to account for the location of the origin:

% \begin{corollary}
% A realization of a combinatorial type is uniquely determined by the edge lengths and vertex distances to some point in the interior.
% \end{corollary}

\begin{corollary}
\label{res:comb_type_rigid}
A polytope with the origin in its interior is uniquely determined~by its face-lattice, its edge lengths and its vertex-origin distances.
%A polytopal realization of a combinatorial type that contains~the~origin in the interior is uniquely determined (up to isometry) by edge lengths and~vertex-origin distances.
\end{corollary}

If the origin lies not in $P$ then a unique reconstruction is not guaranteed (re\-call \cref{fig:origin_outside_tnontriv_ex}).
However, if $0\in\partial P$ then we can say more.
Recall the \emph{tangent cone} of $P$ at a face $\sigma\in\F(P)$: %if $x\in \relint(\sigma)$, then
%
%$$T_P(\sigma):=\{t\in\RR^d\mid \text{$x+\eps t\in P$ for some $\eps>0$}\}.$$
$$T_P(\sigma):=\cone\{x-y\mid x\in P,\, y\in\sigma\}.$$

\begin{theorem}
Let $P,Q\subset\RR^d$ be combinatorially equivalent polytopes with~the~following properties:
\begin{myenumerate}
    \item $0\in\relint(\sigma_Q)$ for some face $\sigma_Q\in\F(Q)$, $\sigma_P$ is the corresponding~face~in~$P$, and $P$ and $Q$ have isometric tangent cones at $\sigma_P$ and $\sigma_Q$.
    % NOTE: we actually need that the isometry respects the combinatorics, but this is technical to state
    \item edges in $Q$ are at most as long as in $P$.
    \item vertex-origin distances in $Q$ are at least as large as in $P$.
\end{myenumerate}
Then $P\cong Q$.
\end{theorem}

%Note that for $\sigma=Q$ or $\sigma$ a facet, the tangent cones are automatically isometric.

Property \itm1 is always satisfied if, for example, $\sigma_Q$ is a facet of $Q$, or if $\sigma_Q$~is~a face of co\-dimension two at which $P$ and $Q$ agree in the dihedral angle.

\begin{proof}
The proof is by induction on the dimension $d$ of the polytopes. The induction base $d=1$ is clearly satisfied.
In the following we assume $d\ge 2$.

%Assume now $d\ge 2$.
Note first that we can apply \cref{thm:main_comb_eq} to $\sigma_P$ and $\sigma_Q$ to obtain $\sigma_P\simeq\sigma_Q$ via an orthogonal transformation $T\:\RR^d\to\RR^d$, in particular $0\in\relint(\sigma_P)\subseteq P$.
By \itm1 this transformation extends to the tangent cones at these faces.
Let $F_{1,Q},...,F_{m,Q}\in\F_{d-1}(Q)$ be the facets of $Q$ that contain $\sigma_Q$, and let $F_{1,P},...,F_{m,P}\in\F_{d-1}(P)$ be the corresponding facets in $P$.
Then $F_{i,P}$ and $F_{i,Q}$ too have isometric tangent cones at $\sigma_P$ \resp\ $\sigma_Q$, and $F_{i,P}\simeq F_{i,Q}$ follows by induction hypothesis.

Now choose a point $x_Q\in\RR^d$ \emph{beyond} the face $\sigma_Q$ (\ie\ above all facet-defining hyper\-planes that contain $\sigma_Q$, and below the others) so that $x_P:=T x_Q$ is beyond the face $\sigma_P$.
Consider the polytopes $Q':=\conv(Q\cup\{x_Q\})$ and $P':=\conv(P\cup\{x_P\})$.
Since $x_Q$ lies beyond $\sigma_Q$, each edge of $Q'$ is either an edge of $Q$, or is an edge between $x_Q$ and a vertex of some facet $F_{i,Q}$; and analogously for $P'$.
The lengths of edges incident to $x_Q$ depend only on the shape of the tangent cone and the shapes of the facets $F_{i,Q}$, hence are the same for corresponding edges in $P'$.
Thus, $P'$ and $Q'$ satisfy the preconditions of \cref{thm:main_comb_eq}, and we have $P'\simeq Q'$.

Finally, as $x_Q\to 0$, we have $Q'\to Q$ and $P'\to P$ (in the Hausdorff metric), which shows that $P\simeq Q$.
\end{proof}

Thus, if the origin lies in the interior of $P$ or a facet of $P$ then \cref{thm:main_comb_eq}~applies. If the origin lies in a face of codimension three, then counterexamples exist.

\begin{example}
Consider the pentagons from \cref{fig:origin_outside_tnontriv_ex} as lying in the plane $\RR^2\times\{1\}$, with their former origins now at $(0,0,1)$. Consider the pyramids 
$$P^*:= \conv(P\cup\{0\})\quad\text{and}\quad Q^*:=\conv(Q\cup\{0\}).$$
These polytopes have the origin in a vertex (a face of codimension three), have the same edge-graphs, edge lengths and vertex-origin distances, yet are not isometric. Examples with the origin in a high-dimensional face of codimension three can be constructed by considering prisms over $P^*$ \resp\ $Q^*$.
\end{example}

We do not know whether \cref{thm:main_comb_eq} holds if the origin lies in a face of codimen\-sion two (see \cref{q:codimension_two}).

% We can conclude that \cref{thm:main_comb_eq} applies unchanged if the origin lies in (the relative interior of) a facet, but fails for a face of codimension 3. 
% It is unclear what happens if the origin lies in a face of codimension 2 (see \cref{q:...}).

\subsection{Inscribed polytopes}
\label{sec:inscribed}

It is worthwhile to formulate versions of \cref{thm:main_comb_eq} for inscribed polytopes, that is, polytopes where all vertices lie on a common sphere -- the circumsphere.
For inscribed polytopes we can write down a direct monotone relation between edge lengths and the circumradius. % (similar to the monotone relation for $\alpha$-expansion in \cref{res:expansion_main_result}).

\begin{corollary}[inscribed version]\label{res:inscribed_version}
Given two combinatorially equivalent polytopes $P,Q\subset\RR^d$ so that
\begin{myenumerate}
    \item $P$ and $Q$ are inscribed in spheres of radii $r_P$ and $r_Q$ respectively,
    \item $Q$ contains its circumcenter in the interior, and
    \item edges in $Q$ are at most as long as in $P$, 
\end{myenumerate}
Then $r_P\ge r_Q$, with equality if and only if $P\simeq Q$.
\begin{proof}
Translate $P$ and $Q$ so that both circumcenters lie at the origin.
Suppose~that $r_P\le r_Q$. Then all preconditions of \cref{thm:main_comb_eq} are satisfied, which yields $P\simeq Q$, hence $r_P=r_Q$.
\end{proof}
\end{corollary}

This variant in particular has already found an application in proving the finitude of so-called \enquote{compact sphere packings} with spheres of only finitely many different radii \cite{kikianty2023compact}.

Interestingly, the corresponding \enquote{unique reconstruction version} does not require any assumptions about the location of the origin or an explicit value for the circum\-radius.
In fact, we do not even need to apply our results, as it already follows from Cauchy's rigidity theorem (\cref{res:Cauchy_rigidity}).

\begin{corollary}
An inscribed polytope of a fixed combinatorial type is uniquely~determined, up to isometry, by its edge lengths.
% An inscribed realization of a combinatorial type is uniquely determined (up to isometry) by its edge lengths.
%
\begin{proof}%[Proof in the appendix (see \cref{res:inscribed_determined_by_edge_length_via_Cauchy})]
The case $d=2$ is straightforward: given any circle, there is only~a~single~way (up to isometry) to place edges of prescribed lengths. Also, there is only~a~single~radius for the circle for which the edges reach around the circle exactly once and close up perfectly.
This proves uniqueness for polygons.

If $P$ is of higher dimension then its 2-dimensional faces are still inscribed, have prescribed edge lengths, and by the 2-dimensional case above, corresponding 2-faces in $P$ and $Q$ are therefore isometric.
Then $P\simeq Q$ follows from Cauchy's~rigidity~theorem (\cref{res:Cauchy_rigidity}).
\end{proof}
\end{corollary}

\section{Conclusion, further notes and many open questions}
\label{sec:conclusions}

We conjectured that a convex polytope is uniquely determined up to isometry~by its edge-graph, edge lengths and the collection of distances between its vertices and some interior~point, across all dimensions and combinatorial types (\cref{conj:main_rigid_intro}).
We~also~posed~a~more general conjecture expressing the idea that polytope skeleta, given their edge lengths, are maximally expanded (\cref{conj:main}).
We developed techniques~based~on~Wachs\-press coordinates and the so-called Izmestiev matrix that led to us to resolve three relevant special cases: centrally symmetric~polytopes (\cref{res:centrally_symmetric}), small perturbations (\cref{thm:main_local}), and combinatorially equivalent polytopes (\cref{thm:main_comb_eq}).
We feel confident that our approach already highlights the essential difficulties in verifying the general case.

%In this section we address various questions and further reaching thoughts that one might have reading the above text.
%We found it is more fitting to collect them here than spreading them throughout the text.

In this section we collected further thoughts on our results, notes on connections to the literature, as well as many questions and future research directions.

%We use this section to comment on various thoughts and ideas surrounding the main conjectures, the results of this article and potential future research directions.
%We found it more fitting to collect them here than to spread them throughout the text, where they might interrupt the reading.

\subsection{Consequences of the conjectures}
\label{sec:vast_generalization}

%We briefly allude to some consequences of \cref{conj:main_rigid_intro}. 

%Foremost, 
\cref{conj:main_rigid_intro} vastly generalizes several known ``reconstruction from the edge-graph'' results.
The following is a special case of \cref{conj:main_rigid_intro}: an inscribed polytopes with all edges of the same length would~be uniquely determined by its edge-graph.
This includes the following special cases:
\begin{itemize}
    \item The reconstruction of matroids from their base exchange graph: a matroid can be identified with its matroid base polytopes, which is a 01-polytopes (hence inscribed) and has all edges of length $\sqrt 2$.
    This reconstruction has been initially proven in \cite{holzmann1973graphical} and recently rediscovered in \cite{pineda2022reconstructibility}.
    \item The reconstruction of simultaneously vertex- and edge-transitive polytopes from their edge-graph: this was proven in \cite{winter2020symmetric,winter2021spectral}, essentially using the tools of this article.
\end{itemize}
It would imply an analogous reconstruction from the edge-graph for classes of~polytopes such as the uniform polytopes or higher-dimensional inscribed Johnson~solids \cite{johnson1966}.

Secondly, a positive answer to \cref{conj:main_rigid_intro} would also resolve Question~6.6~in \cite{winter2021capturing} on whether the metric coloring can capture the Euclidean symmetries~of~a~polytope.

% Lastly, \cref{conj:main} would generalizes the \emph{finite dimensional Kirszbraun theorem} \cite{kirszbraun1934zusammenziehende}.
% The Kirszbraun theorem states that if pairwise distances between points in a set of points are shrinking, then the set becomes smaller in the sense that all points come closer to some ...

\subsection{\Cref{conj:main} for graph embeddings}
\label{sec:fixing_q_version_conjecture}

In \cref{ex:cube} we show that~\cref{conj:main} does not hold when replacing $Q$ by some more general graph embedding $q\:V(G_P)\to\RR^e$ of $G_P$, even if $0\in\Int\conv(q)$.

%We believe that the underlying reason 

Our intuition for why this fails, and also what distinguishes it from the setting of our conjectures and the verified special cases is, that the embedding of \cref{ex:cube} does not ``wrap around the origin'' properly.
It is not quite clear what this means for an embedding of a graph, except that it feels right to assign this quality to polytope skeleta, to embeddings close to them, and also to centrally symmetric embeddings.

One possible formalization of this idea is expressed in the conjecture below, that is even stronger than \cref{conj:main} (the idea is due to Joseph Doolittle):

% Potentially interesting conjectures can however be formulated by imposing suitable constraints on $q$, such as either of the following:
% %
% \begin{myenumerate}
%     \item for each vertex $i\in V(G_P)$ the cone 
%     %
%     $$q_i + \cone\{q_j-q_i\mid ij\in E(G_P)\}$$
%     %
%     contains the origin (suggested by Joseph Doolittle).
%     \item the framework consisting of the edges of $q$ and the central rods (connecting the vertices and the origin) has a non-zero self-stress.
% \end{myenumerate}
% %
% The intuition behind these conditions is to formalize the idea of $q$ \enquote{wrapping around the origin} in the same way that a polytope does. Note that all three, being centrally symmetric, being close to a polytope, and being a polytope skeleton achieve this.

\begin{conjecture}\label{conj:josephs_conjecture}
%\begin{conjecture}\label{conj:main}
Given a polytope $P\subset\RR^d$ and a graph embedding $q\:V(G_P)\to\RR^e$ of its edge-graph $G_P$, so that
\begin{myenumerate}
    %\item $0\in \Int\conv(q)$,
    \item for each vertex $i\in V(G_P)$ the cone 
    $$C_i:=q_i + \cone\{q_j-q_i\mid ij\in E(G_P)\}$$
    contains the origin in its interior,
    \item edges in $q$ are at most as long as in $P$, and
    \item vertex-origin distances in $q$ are at least as large as in $P$,
\end{myenumerate}
then $\skel(P)\simeq q$.
\end{conjecture}

Note that since $\bigcap_i C_i\subseteq\conv(q)$, \itm1 already implies $0\in\Int\conv(q)$.

\subsection{Classical rigidity of frameworks}
\label{sec:classical_rigidity}

We previously commented on a \mbox{natural~in}\-terpretation of \cref{res:centrally_symmetric} and \cref{thm:main_local} in the language of classical~rigidity theory  (we refer to \cite{connelly1996second} for any rigidity specific terminology used below).

Consider the edges~of $P$ as \emph{cables} that can contract but not expand, and connect all vertices of $P$ to the origin using \emph{struts} that can expand but not contract.
%The~bars are connected at the origin and the vertices by flexible joints.
This is known as a \emph{tensegrity framework}, and we shall call it the \emph{tensegrity} of $P$.
\Cref{thm:main_local} then asserts that these tensegrities are always (locally) rigid.

Using the language of rigidity, a number of natural follow up questions arise.\nls
So it turns out that swapping cables and struts does not necessarily~preserve rigidity; see \Cref{fig:twisting_cube} for an example.
As a consequence, the tensegrity of a polytope is not necessarily \emph{infinitesimally rigid}, because infinitesimally rigid frameworks stay rigid under swapping cables and struts.

% It is natural to ask whether this tensegrity framework keeps its rigidity when~we instead put struts at the edges and cables between vertices and the origin.
% \Cref{fig:twisting_cube} shows that this can fail.
% Since an \emph{infinitesimally rigid} framework keeps its rigidity on swapping cables and struts, we see that such a polytopal framework~is~not~necessarily infinitesimally rigid.

\begin{figure}[h!]
    \centering
    \includegraphics[width=0.7\textwidth]{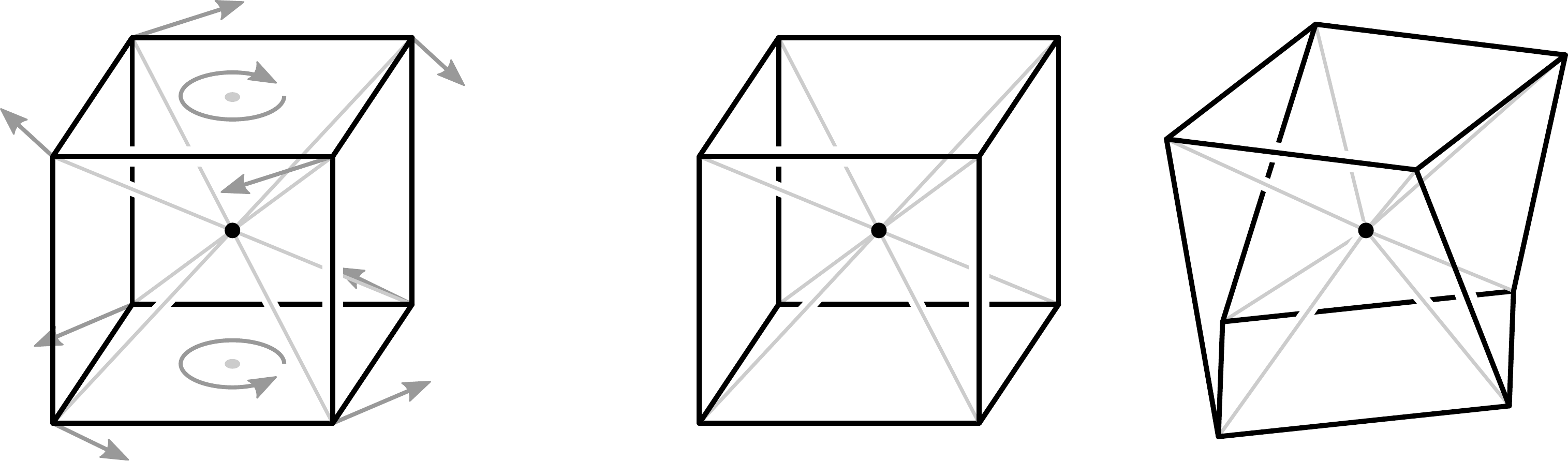} 
    \caption{Already the skeleton of the cube is not rigid if considered as a tensegrity framework with struts for edges and central cables. Twisting the top and bottom face lengthens the edge struts but keeps the central cables of a fixed length. This corresponds to the infinitesimal flex shown on the left.}
    \label{fig:twisting_cube}
\end{figure}

Lacking first-order rigidity, we might ask for higher-order rigidity instead: % (we refer to \cite{connelly1996second} for the terminology):

\begin{question}
    Is the tensegrity of a polytope always \emph{second-order rigid}, or perhaps even \emph{prestress stable}?
\end{question}

For an interpretation of \cref{res:centrally_symmetric} as a tensegrity framework, consider a cable at each edge as before, but each central strut now connects a vertex $p_i$ to~its antipodal counterpart $-p_i$, and is fixed in its center to the origin.
\Cref{res:centrally_symmetric} then asserts that this tensegrity framework is \emph{universally rigid}, \ie\ it has a unique~realization across all dimensions.

Here too, swapping cables and struts does not preserve universal or even global rigidity (see \cref{fig:octagon_ex}).
It does not preserve local rigidity either (see \cref{ex:twisting_4_cube}).
%We do however not know whether it preserves (local) rigidity, and we do not know whether such a central symmetry-forced tensegrity framework (or~cs-frameworks for short) is infinitesimally rigid.

\begin{figure}[h!]
    \centering
    \includegraphics[width=0.37\textwidth]{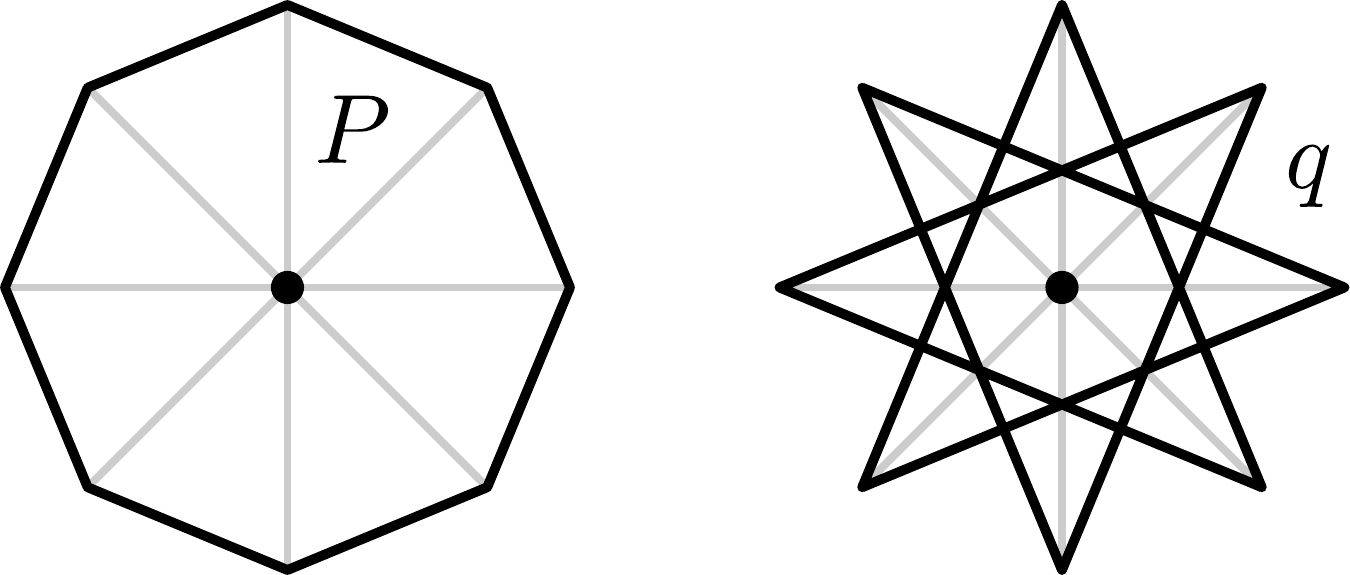}
    \caption{An octagon and an embedding of its edge-graph with longer edges but equally long central cables, showing that the respective tensegrity framework is not globally rigid under forced central symmetry.}
    \label{fig:octagon_ex}
\end{figure}

% \begin{question}
% Is the cs-framework of a polytope infinitesimally rigid?
% \end{question}

% \begin{question}
% Is the cs-framework with struts for edges and cables connecting~vertices to their antipodal counterpart through the origin (locally) rigid?
% \end{question}

\begin{example}
\label{ex:twisting_4_cube}
Consider the 4-cube with its ``top'' and ``bottom'' facets (which are~3-cubes) embedded in the hyperplanes $\RR^3\times\{\pm 1\}$ respectively.
We flex the~skeleton~as follows: deform the~top facet as shown in \cref{fig:twisting_cube}, and the bottom~facet~so~as~to keep the framework centrally symmetric, while keeping both inside their respective hyperplanes.
The edge struts~inside the facets become longer, and the edge~struts~between the facets have previously been of minimal length between the hyperplanes, can therefore also only increase in length.
The lengths of the central cables stay~the same.
\end{example}

As a consequence, the centrally symmetric tensegrity frameworks too are not~necessarily infinitesimally rigid.

%\msays{cite \cite{connelly2006stress} Theorem 4.1 and Theorem 5.1}

% The \emph{cs-framework} of an origin symmetric polytope $P\subset\RR^d$ consist of bars for the edges and central bars \emph{through} the origin, connecting antipodal vertices but also being fixed to the origin in their middle. 
% The respective tensegrity replaces edge bars by cables and central bars by struts.
% \cref{res:centrally_symmetric} establishes that this framework is universally rigid.

%Results like \cref{res:centrally_symmetric} are classically understood as \emph{universal rigidity}, fixing a framework uniquely across all dimensions.

%In contrast, \cref{sec:local_rigidity} established local rigidity, allowing for further solutions in the same or higher dimension, but not in a neighborhood of $\skel(P)$.

%We wonder about even weaker forms of rigidity:

\subsection{Schlegel diagrams}
\label{sec:Schlegel_diagram}

Yet another interpretation of the frameworks discussed in \cref{sec:classical_rigidity} is as skeleta of special \emph{Schlegel diagrams}, namely, of pyramids whose base facet is the polytope $P$. It is then natural to ask whether a general Schlegel diagram is rigid as well (this was brought up by Raman Sanyal).

\begin{figure}[h!]
    \centering
    \includegraphics[width=0.42\textwidth]{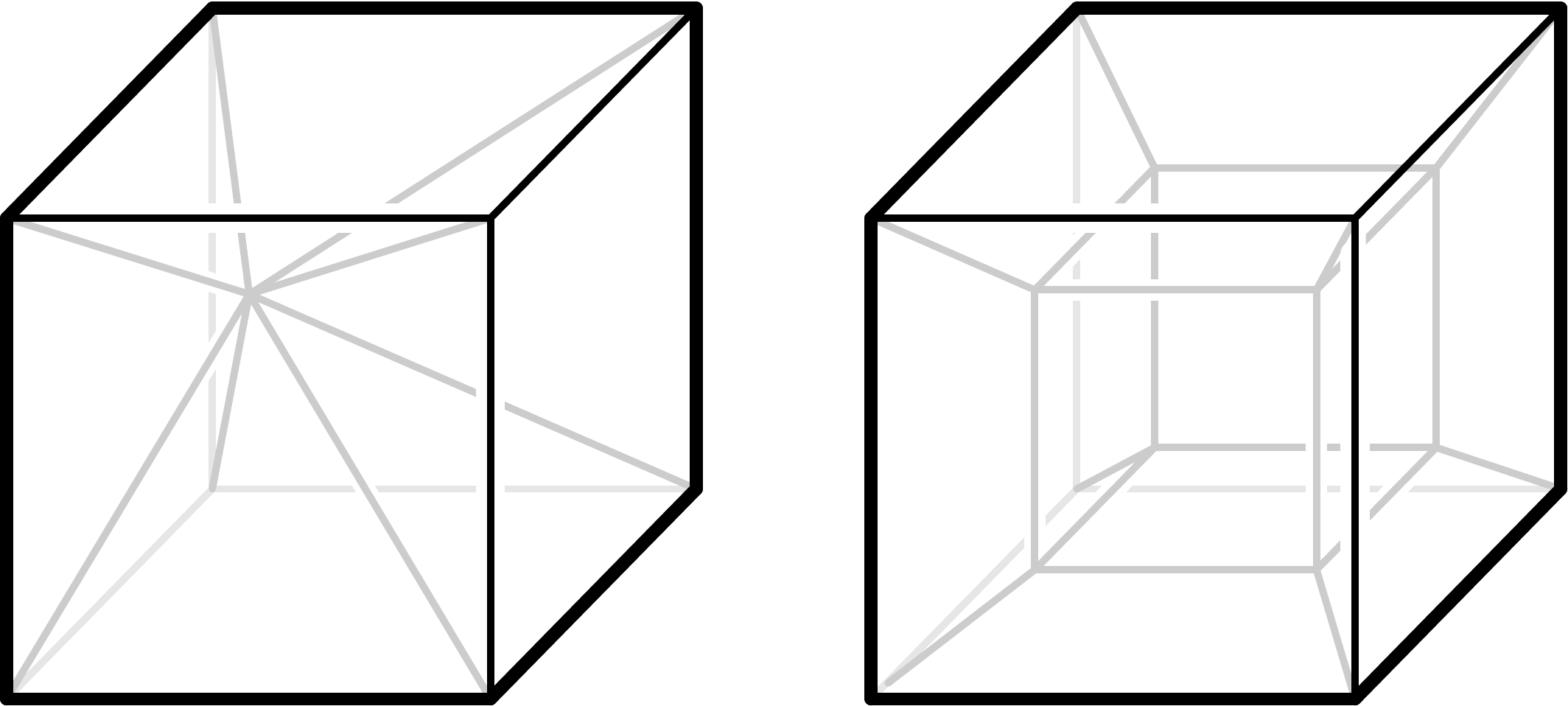}
    \caption{Two Schlegel diagrams of 4-polytopes: of the pyramid with the 3-cube as base facet (left) and of the 4-cube (right).}
    \label{fig:Schlegel_diagram}
\end{figure}

The question of rigidity for Schlegel diagrams is already interesting in dimension two, that is, for Schlegel diagrams of 3-polytopes. 
The edge-graphs of many 3-poly\-topes are too sparse to be generically rigid in $\RR^2$, and~so one might expect that~most of their Schlegel diagrams are flexible.
Indeed, flexible Schlegel diagrams exist (see \cref{fig:flex_Schlegel}, left).

\begin{figure}[h!]
    \centering
    \includegraphics[width=0.67\textwidth]{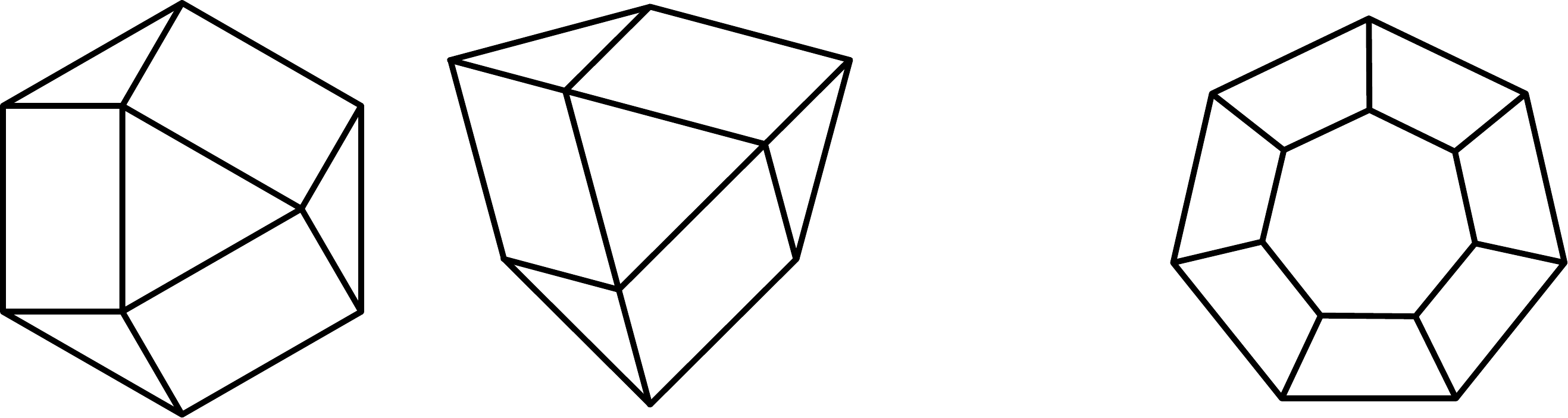}
    \caption{A flexible Schlegel diagram (left), and a rigid Schlegel diagram (right).}
    \label{fig:flex_Schlegel}
\end{figure}

Surprisingly however, this seems to be the exception rather than the rule.
For example, we believe that Schlegel diagrams of $(2n+1)$-gonal prisms are always rigid (see \cref{fig:flex_Schlegel}, right).
Since Schlegel diagrams are very special realizations (they are projections of convex objects), the generic ones among them might very well be rigid. This is not clear so far.

\begin{question}
    Is a generic Schlegel diagram rigid?
\end{question}

\iffalse % (porbably wrong) conjectures about Schlegel diagrams

We list some questions that come to mind, but we do not feel comfortable~to~state any conjectures.
The following is an analogue of \cref{conj:main_rigid_intro}:

\begin{question}
\label{q:Schlegel_reconstruction}
Let $\mathcal P$ and $\mathcal Q$ be Schlegel diagrams of polytopes with the same~edge-graph and edge lengths. Is then $\mathcal P\simeq \mathcal Q$?
\end{question}

The following stronger version is an analogue of \cref{conj:main}:

\begin{question}
\label{q:Schlegel_expansion}
Let $\mathcal P$ and $\mathcal Q$ be Schlegel diagrams of polytopes with the same~edge-graph and and so that
%
\begin{myenumerate}
    \item boundary edges in $\mathcal Q$ are at most as long as in $\mathcal P$, and
    \item inner edges in $\mathcal Q$ are at least as long as in $\mathcal P$.
\end{myenumerate}
%
Do we have $\mathcal P\simeq\mathcal Q$?
\end{question}

\begin{question}
\label{q:Schlegel_locally_rigid}
Consider the skeleton of a Schlegel diagram $\mathcal P$ as a bar-joint framework. Is it (locally) rigid?
\end{question}

\fi

Above we considered Schlegel diagrams as bar-joint frameworks. 
If we consider them as tensegrity frameworks then it is easy to find generically flexible examples (see \cref{fig:Schlegel_diagram_and_twist}).
Schlegel diagrams are also not necessarily globally rigid (see \cref{fig:Schlegel_diagram_cube_folding}).
% (for example, the Schlegel~diagram of the 3-cube can be ``folded'' along a diagonal).

\begin{figure}[h!]
    \centering
    \includegraphics[width=0.37\textwidth]{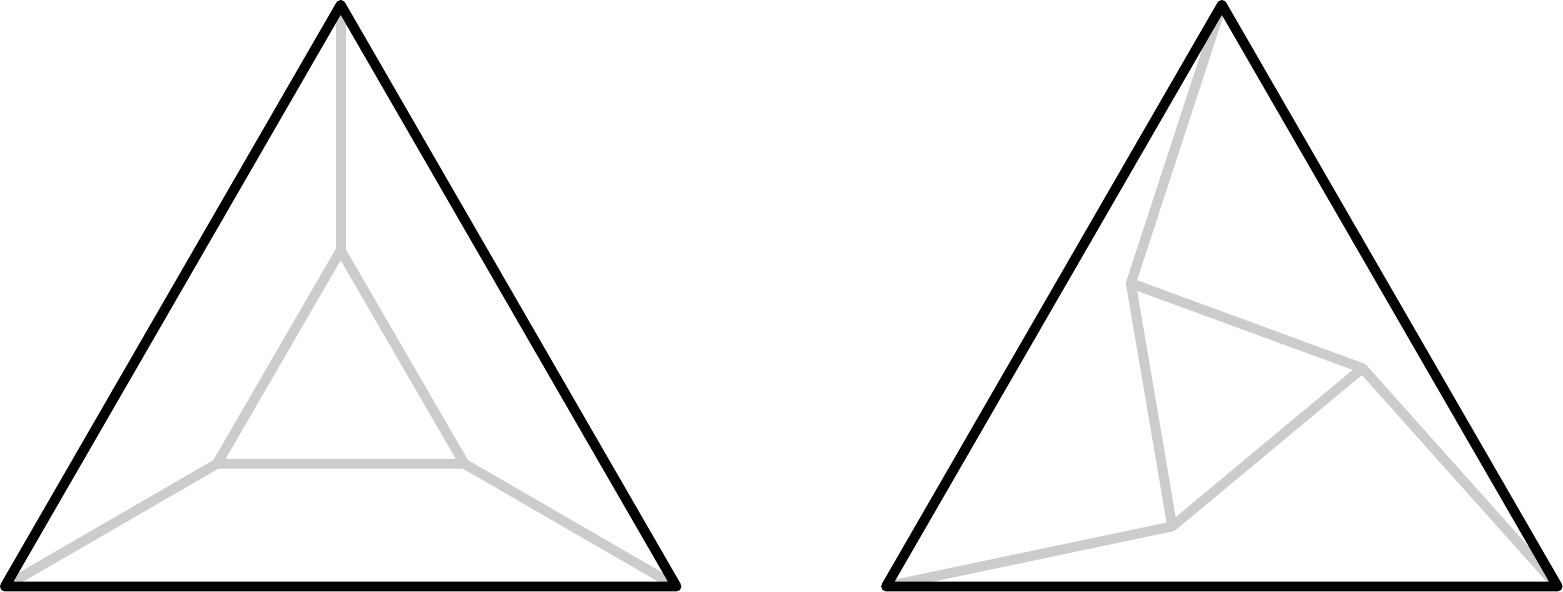}
    \caption{The skeleton of Schlegel diagram of a triangular prism with cables on the outside and struts on the inside is not rigid. Twisting the inner triangle increases the lengths of struts and fixes all other lengths.}
    \label{fig:Schlegel_diagram_and_twist}
\end{figure}

\begin{figure}[h!]
    \centering
    \includegraphics[width=0.37\textwidth]{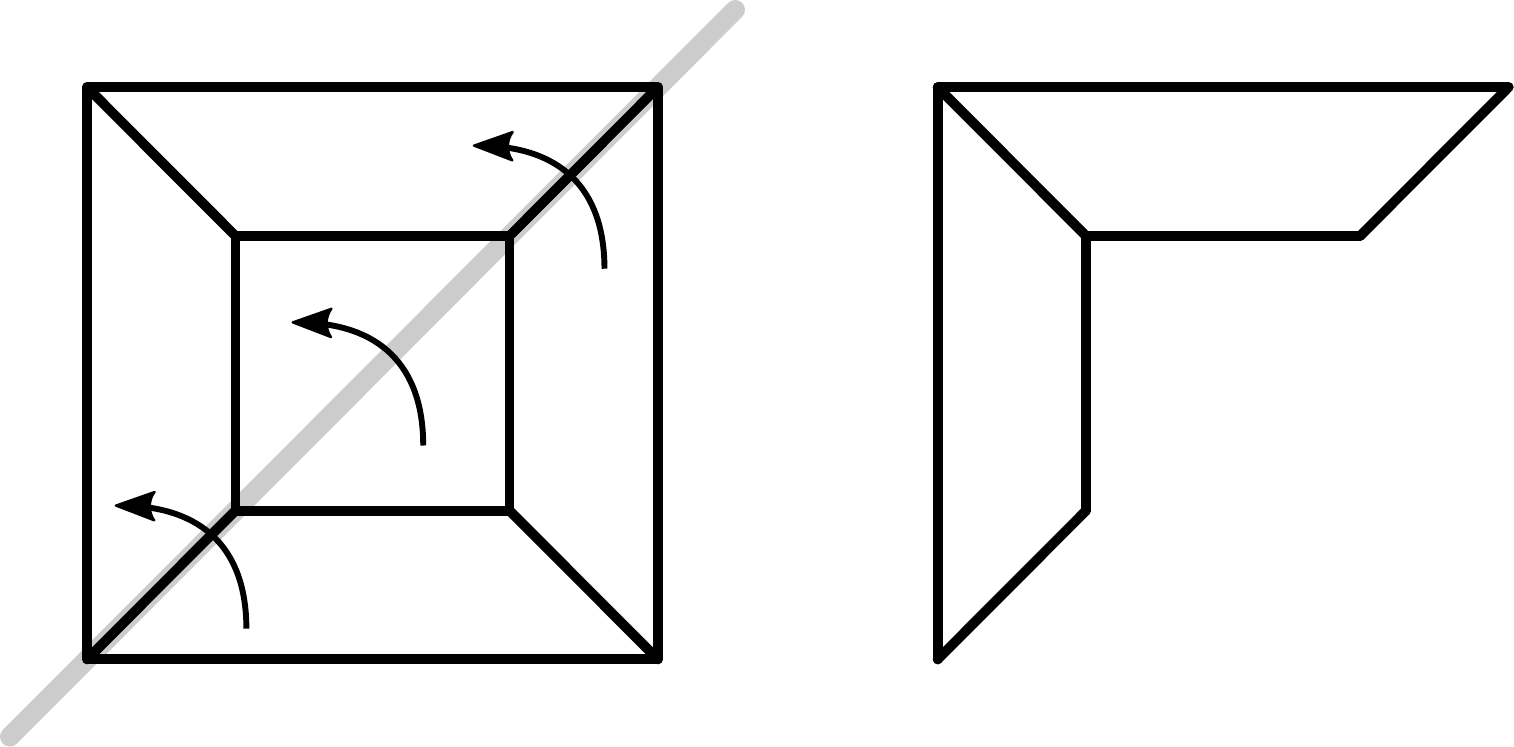}
    \caption{Folding the Schlegel diagram of the 3-cube along a diagonal preserves all edge lengths.}
    \label{fig:Schlegel_diagram_cube_folding}
\end{figure}

\iffalse % comment on Connelly's comments + another question

\begin{question}
Is a polytope uniquely determined by its edge-graph and the length of edges in some Schlegel diagram?    \TODO
\end{question}

% \hrulefill

% \begin{conjecture}
% The combinatorial type of a polytope is determined by its edge-graph and the lengths of the edges in one Schlegel diagram.
% \end{conjecture}

% \begin{conjecture}
% The combinatorial type of a polytope is determined by its edge-graph and its edge lengths.
% \end{conjecture}

It was recognized by Robert Connelly and Zhen Zhang (personal communication) that many Schlegel diagrams are \emph{not} rigid if considered as tensegrities with cables for the outer edges: consider ...
In contrast, for struts on the outside they provided computations showing that the standard Schlegel diagram of the 4-cube is not only rigid, but even prestress stable.
This is perhaps the better question to ask: ...
\end{disable}

\fi

\subsection{Stoker's conjecture}
\label{sec:Stokers_conjecture}

Stoker's conjecture asks whether the dihedral angles of a polytope determine its face angles, and thereby its overall shape to some degree.
Recall that \emph{dihedral angles} are the angle at which facets meet in faces of codimension two, whereas \emph{face angles} are the dihedral angles of the facets.
Stoker's~conjecture was asked in 1968 \cite{stoker1968geometrical}, and a proof was claimed recently by Wang and Xie \cite{wang2022gromov}:

\begin{theorem}[Wang-Xie, 2022]\label{res:Stokers_conjecture}
Let $P_1$ and $P_2$ be two combinatorially equivalent polytopes such~that corresponding dihedral angles are equal.
Then all corresponding face angles are equal as well.
\end{theorem}

Our results allow us to formulate a semantically similar statement.
%
%A semantically related statement follows already from the results of this article.
The following is a direct consequence of \cref{res:comb_type_rigid} when expressed for the polar dual polytope:

\begin{corollary}
\label{res:Stokers_conjecture_variant}
Let $P_1$ and $P_2$ be two combinatorially equivalent polytopes such that corresponding dihedral angles and facet-origin distances are equal.
Then $P_1\simeq P_2$.
\end{corollary}

While the assumptions in \cref{res:Stokers_conjecture_variant} are unlike stronger compared~to~Sto\-ker's conjecture (we require facet-origin distances), we also obtain isometry instead of~just identical face angles.
While related, we are not aware that either of \cref{res:Stokers_conjecture}~or \cref{res:Stokers_conjecture_variant} follows from the other one easily.

\subsection{Pure edge length constraints}

Many polytopes cannot be reconstructed~up to isometry from their edge-graph and edge lengths alone (recall \cref{fig:flex_polytope}).
However, for all we know the following is open:
%Still, the following remains open:

\begin{question}
    \label{q:comb_from_edge_lengths}
    Is the combinatorial type of a polytope uniquely determined by its edge-graph~and edge lengths?
\end{question}

This alone would already prove \cref{conj:main_rigid_intro} (by \cref{res:comb_type_rigid}).
It would~also imply a positive answer to the following: % question asking for a strengthening of Cauchy's rigidity theorem (\cf\ \cref{res:Cauchy_rigidity}):

\begin{question}
    Is a polytope uniquely determined up to isometry by its 2-skeleton (\ie\ the face-lattice cut off at, but including dimension two) and the shape of each 2-face?
\end{question}

Note that this is a particular strengthening of Cauchy's rigidity theorem, which requires the face-lattice to be prescribed in its entirety, rather than on some lower levels only.

Let us now fix the combinatorial type.
We are aware of three types of polytopes that are not determined (up to isometry) by their face-lattice and edge lengths:
\begin{myenumerate}
    \item $n$-gons with $n\ge 4$.
    \item Minkowski sums: if $P=Q+R$ and $Q$ and $R$ are generically oriented \wrt\ each other, then a slight reorientation of the summands changes the shape of $P$ but keeps its edge lengths (see \cref{fig:deforming_cuboctahedron}).
    \item polytopes having all edge~directions on a ``conic at infinity'': this implies~an affine flex (see \cite{connelly2018affine}). This is most easily implemented for~zonotopes (\mbox{recall}~\cref{fig:length_preserving_linear_trafo}), but happens for other polytopes as well, such as~3-polytopes with~up to five edge directions (see \cref{fig:sliced_cube}).
\end{myenumerate}

\begin{figure}[h!]
    \centering
    \includegraphics[width=0.46\textwidth]{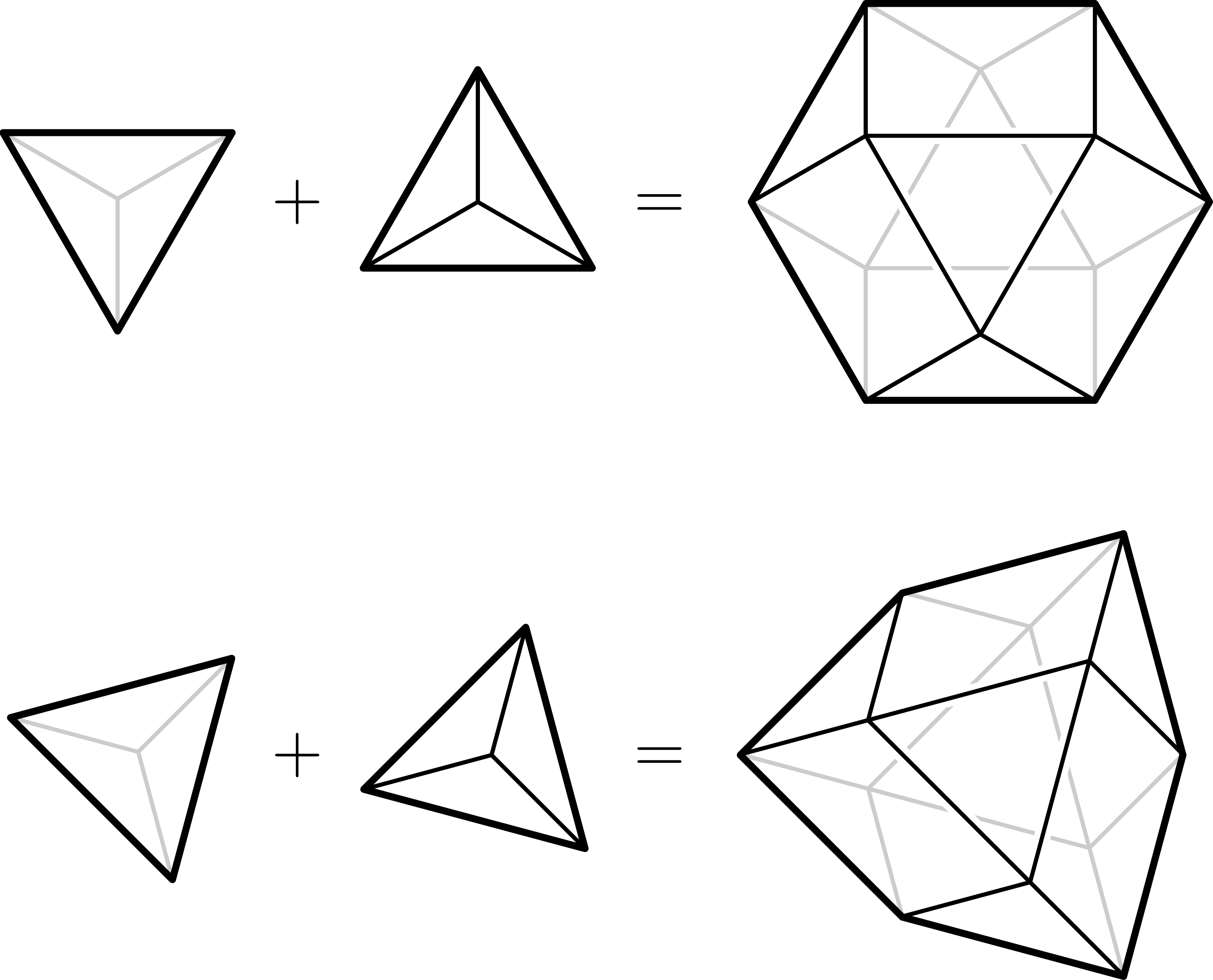}
    \caption{The cuboctahedron can be written as the Minkowski sum~of two simplices, and twisting these simplices leads to a flex of the cuboctahedron that preserves edge lengths.}
    \label{fig:deforming_cuboctahedron}
\end{figure}

\begin{figure}[h!]
    \centering
    \includegraphics[width=0.16\textwidth]{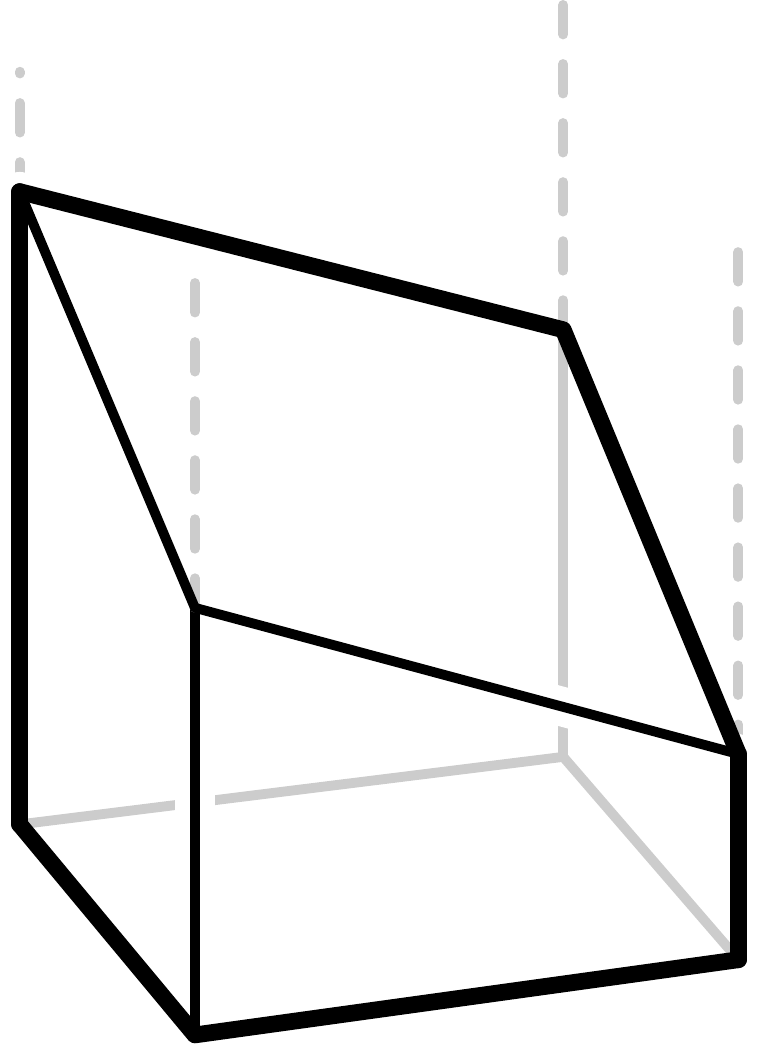}
    \caption{A cuboid sliced at an angle in an appropriate way has only five edge directions and has an edge length preserving affine flex. The flex deforms the bottom face into a rhombus and keeps the vertical edges vertical.}
    \label{fig:sliced_cube}
\end{figure}

We are not aware of other examples of polytopes that flex in this way and~so~we wonder whether this is already a full characterization.

\begin{question}
If a polytope is \emph{not} determined up to isometry by its combinato\-rial type and edge lengths, is it necessarily a polygon, a non-trivial Minkowski sum or has all its edge directions on a conic at infinity? 
Is this true at least~up~to~dimen\-sion three?
\end{question}

In how far a 3-polytope is determined by local metric data at~its edges
%The question of flexibility of 3-polytopes 
was~reportedly discussed in an Oberwolfach question session (as communicated by Ivan~Izmes\-tiev on MathOverflow \cite{434771}), where the following more general question was asked:

\begin{question}\label{q:Ivans_question}
Given a simplicial 3-polytope and at each edge we prescribe either the length or the dihedral angle, in how far does this determine the polytope?
\end{question}

Having length constraints at every edge determines a simplicial polytope already up to isometry via Cauchy's rigidity theorem (\cref{res:Cauchy_rigidity}).
The angles-only version is exactly the 3-dimensional Stoker's conjecture (\cref{sec:Stokers_conjecture}). We are not~aware that this question has been addressed in the literature beyond these two extreme cases.

Note also that \cref{q:Ivans_question} is stated for \emph{simplicial} 3-polytopes, but actually~includes general 3-polytopes via a trick: if $P$ is not simplicial, triangulate every 2-face, and at each new edge created in this way prescribe a dihedral angle of $180^\circ$ to prevent~the faces from folding at it.

\subsection{Injectivity of the Wachspress map $\boldsymbol{\phi}$}
\label{sec:Wachspress_map}

%For combinatorially equivalent~polytopes $P,Q\subset\RR^d$ recall the \emph{Wachspress map} $\phi\:P\to Q$ (\cref{def:Wachspress_map}) which sends a point~$x\in P$ with Wachspress coor\-dinates $\alpha\in\Delta_n$ onto $\sum_i \alpha_i q_i\in Q$.

In \cref{res:Wachspress_map_surjective} we proved that~the Wachspress map $\phi\: P\to Q$ (\cf\ \cref{def:Wachspress_map}) between combinatorially equivalent polytopes is surjective.
%$\phi$ is surjective.
In contrast, the injectivity of~the Wachspress map has been established only in dimension two by Floater and Kosinka \cite{floater2010injectivity} and is conjectured for all $d\ge 3$.

\begin{conjecture}
\label{conj:Wachspress_injective}
The Wachspress map $\phi\: P\to Q$ is injective.
\end{conjecture}

%The injectivity of the Wachspress map has been established only in dimension~two by Floater and Kosinka \cite{floater2010injectivity} and is conjecture for all $d\ge 3$.
If true, the Wachspress map would provide an interesting and somewhat canonical homeomorphism (in fact, a rational map, see \cite{warren2003uniqueness}) between any two combinato\-rially~equivalent polytopes.

%\subsection{What if $\boldsymbol{0\not\in\Int(Q)}$?}
\subsection{What if $\boldsymbol{0\not \in \mathrm{int}(Q)}$?}

If $0\not\in Q$ then \cref{fig:origin_outside_trivial_ex,fig:origin_outside_tnontriv_ex} show that our conjectures fail. 
We do however not know whether in the ``unique reconstruction''~case~the~num\-ber of solutions would be finite.

\begin{question}
    Given edge-graph, edge lengths and vertex-origin distances, are there only finitely many polytopes with these parameters? %\msays{Or is this easy?}
\end{question}

This is in contrast to when we replace $Q$ with a graph embedding $q\:V(G_P)\to\RR^e$, which can have a continuum of realizations (see \cref{fig:flex_pyramid}).

\begin{figure}[h!]
    \centering
    \includegraphics[width=0.4\textwidth]{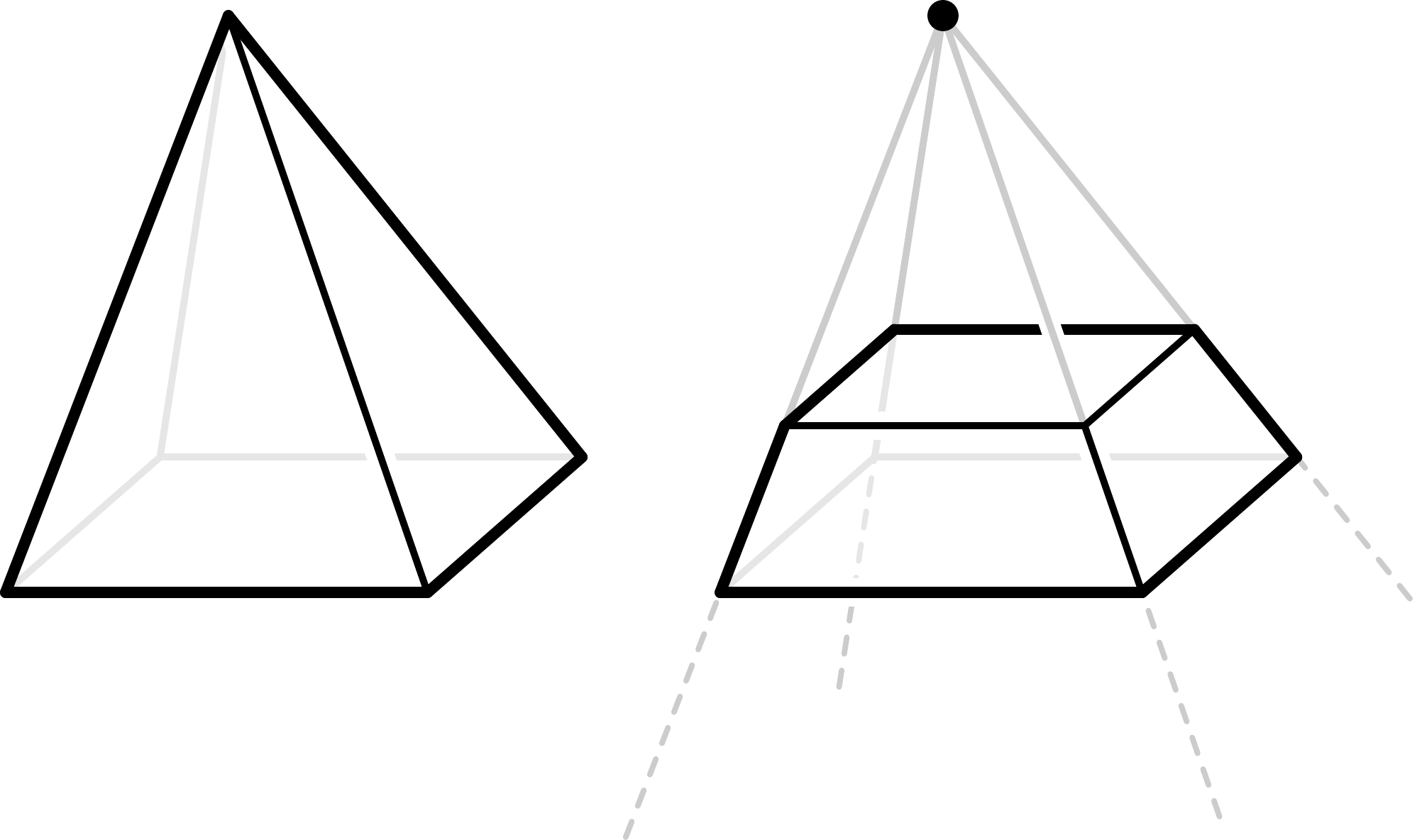} 
    \caption{The square based pyramid (left) is flexible as a framework (since then the bottom face needs not stay flat). Likewise, the framework of the square based frustum with this particular choice of origin (right) flexes. It is however (locally) rigid as a polytope.}
    \label{fig:flex_pyramid}
\end{figure}

In \cref{sec:combinatorial_equivalence} we showed that reconstruction from the face-lattice, edge lengths and vertex-origin distances is possible even if the origin lies only in the inside of a facet of $P$, but that it can fail if it lies in a face of codimension \emph{three}. 
We do not know what happens for a face of codimension \emph{two}.

\begin{question}
\label{q:codimension_two}
Is a polytope
uniquely determined by its face-lattice, edge lengths and vertex-origin distances if the origin is allowed to lie in the inside of faces of~co\-dimension $0, 1$ and $2$?
%     %Are there counterexamples where the origin lies in a face of codimension two?
% %
%     Do \cref{conj:main_rigid_intro} and \cref{conj:main} generalize to $0\in\relint(\sigma)$, where $\sigma\in\mathcal F_{d-2}(Q)$ is a face of codimension two?
\end{question}

\newpage

\appendix

\section{Perron-Frobenius theory}
\label{sec:appendix_perron_frobenius}

The following fragment of the Perron-Frobenius theorem is relevant to this article. Recall that a matrix is \emph{irreducible} if no simultaneous row-column permutation brings it in a block-diagonal form with more than one block; or equivalently, if it is not the (weighted) adjacency matrix of a disconnected graph.
See also \cite{frobenius1912matrizen}.

\begin{theorem}[Perron-Frobenius]\label{res:Perron_Frobenius}
Let $M\in\RR^{n\x n}$ be a non-negative irreducible symmetric matrix, then 
\begin{myenumerate}
    \item the largest eigenvalue $\theta$ of $M$ is positive and has multiplicity one.
    \item there is a $\theta$-eigenvector $z\in\RR^n$ with strictly positive entries. 
\end{myenumerate}
\end{theorem}

\section{Cauchy's rigidity theorem}
\label{sec:appendix_cauchy_rigidity}

Cauchy's famous rigidity theorem was initially formulated in dimension three~and is often quoted briefly as follows:
\begin{quote}
    \emph{3-polytopes with isometric faces are themselves isometric}.
\end{quote}
Generalizations to higher dimensions have been proven by Alexandrov \cite{alexandrov2005convex} where~one assumes isometric \emph{facets} to conclude global isometry (see also its proof in \cite{pak2010lectures}).
We state a rigorous version that only requires isometric 2-faces and that can be easily derived from the facet versions using induction by dimension:

\begin{theorem}[Cauchy's rigidity theorem, version with 2-faces]\label{res:Cauchy_rigidity}
Given two combinatorially equivalent polytopes $P,Q\subset\RR^d$ and a face-lattice isomorphism $\phi\:\F(P)\to\F(Q)$.
If $\phi$ extends to an isometry on every 2-face $\sigma\in\F_2(P)$, then $\phi$ extends to an isometry on all of $P$, that is, $P\simeq Q$.
\end{theorem}

% \begin{proof}
%\TODO
% \end{proof}

% \begin{corollary}\label{res:inscribed_determined_by_edge_length_via_Cauchy}
% % Given two combinatorially equivalent \ul{inscribed} $P,Q\subset\RR^d$ and a face lattice isomorphism $\phi\:\F(P)\to\F(Q)$.
% % If each edge $e\in\F_1(P)$ of $P$ is as long as its counterpart $\phi(e)$ in $Q$, then $P\cong Q$.
% An inscribed polytope of a fixed combinatorial type is uniquely~determined by its edge lengths.
% %
% \begin{proof}
% This is easy to show for 2-dimensional polytopes.
% If $P$ is of higher dimension then its 2-faces are still inscribed.
% By the prescribed edge lengths, corresponding 2-faces in $P$ and $Q$ are then isometric.
% $P\simeq Q$ follows from \cref{res:Cauchy_rigidity}.
% \end{proof}
% \end{corollary}

% Note that this version does not require that the circumcenter is contained in any of the involved polytopes.

\section{Some linear algebra}
\label{sec:appendix_linear_algebra}

\begin{theorem}\label{res:linear_algebra}
Given two matrices $A\in\RR^{d\x n}$ and $B\in\RR^{d\x m}$ \mbox{with $\Span B\subseteq\Span A$}, there exists a linear transformation $T\:\RR^n\to\RR^m$ with $T A\T = B\T$.
\begin{proof}
Set $U_A:=\Span A$ and $d_A:=\dim U_A=\rank A$. 
Respectively, set $U_B:=\Span B\subset U_A$ and $d_B:=\dim U_B=\rank B$.
We can assume that the columns of $A$ and $B$ are sorted so that $a_1,...,a_{d_A}\in U_A$ form a basis, and likewise $b_1,...,b_{d_B}\in U_B$ form a basis. Let $\tilde T\:U_A\to U_B$ be the uniquely determined linear map that maps $\tilde T  a_i = b_i$ for $i\in\{1,...,d_B\}$ and $\tilde T a_i = 0$ for $i\in\{d_B+1,...,d_A\}$. Then $\tilde T A = B$.

%Recall the notion of pseudo inverse. 
The Moore-Penrose pseudo inverse $A^\dagger\in\RR^{n\x d}$ of $A$ satisfies $AA^\dagger=\pi_{U_A}$, where $\pi_{U_A}$ is the orthogonal projection onto $U_A$.
We set $T:= (A^\dagger \tilde T A)\T$ and verify
$$TA\T = (AT\T)\T = (AA^\dagger\tilde T A)\T = (\pi_{U_A} \tilde T A)\T = (\pi_{U_A} B)\T = B\T,$$
where for the last equality we used that all columns of $B$ are already in $U_A$ and the projection acts as identity.
\end{proof}
\end{theorem}

%\newpage
\section{A topological argument}
\label{sec:topological_argument}

% Let $B^d \subset\Bbb R^d$ denote the $d$-dimensional unit ball and $S^d:=\partial B^d$ the bounding sphere.

% \begin{theorem}
% Let $f\:B^d\to\RR^d$ be a continuous map with $0\not\in f(\partial B^d)$. If the~restriction $f|_{\partial B^d}$ is homotopic to the identity $\id\:\partial B^d\to \partial B^d$ in $\RR^d\setminus\{0\}$, then $0\in\Int(B^d)$.
% \end{theorem}
% %
% \begin{proof}
% \TODO
% \end{proof}

% \begin{theorem}
% Given a continuous mapping $f\:B^d\times[0,1]\to\RR^d$ as well as a point $x\in\Int f(B^d,0)$ so that 
% %
% \begin{myenumerate}
%     \item $f(\free,0)$ is injective, and
%     \item the restriction $\partial f\: \partial B^d\times[0,1]\to\RR^d$ is a homotopy in $\RR^d\setminus\{x\}$,
% \end{myenumerate}
% %
% then $x\in\Int f(B^d,1)$ as well.
% \end{theorem}
% %
% \begin{proof}
% Suppose that $0\not\in \Int f(B^d,1)$. Since $\partial f$ is a homotopy of $\partial B^d$ in $\RR^d\nozero$, we actually have $0\not\in f(B^d,1)$.

% Without loss of generality, assume that $x=0$.

% Let $\rho\:\RR^d\nozero\to\partial B^d$ be the continuous projection $y\mapsto y\|y\|$.
% \end{proof}

\begin{lemma}\label{res:topological_argument}
Let $K\subset\RR^d$ be a compact convex set, $x\in\Int(K)$ a point and $f\:K\times[0,1]\to\RR^d$ a homotopy with $f(\free,0)=\id_K$. If the restriction $f|_{\partial K}\: \partial K\times[0,1]\to\RR^d$ yields a homotopy of $\partial K$ in $\RR^d\setminus\{x\}$, then $x\in\Int f(K,1)$.
\end{lemma}
\begin{proof}
Suppose that $x\not\in\Int f(K,1)$.
Since $\partial f$ is a homotopy in $\RR^d\setminus\{x\}$, we actually have $x\not\in f(K,1)$.
We derive a contradiction.

Construct a map $g\: K\to\partial K$ as follows: for $y\in K$ consider the unique ray~emanating from $x$ passing through $f(y,1)$. 
Let $g(x)$ be the unique intersection of this ray with $\partial K$.
Likewise, construct the map $h\:\partial K\times[0,1]\to\partial K$: for $y\in\partial K$ and $t\in[0,1]$, let $h(y,t)$ be the intersection of $\partial K$ with the unique ray emanating from $x$ and passing through $f(y,t)$. 
Note that $h(\free,0)=\id_{\partial K}$ and $h(\free,1)= g|_{\partial K}$.
In other words, $g|_{\partial K}$ is homotopic to the identity on $\partial K$.

The existence of such a map $g\:K\to\partial K$ is a well-known impossibility. This can be quickly shown by considering the following commutative diagram (left) and the diagram induced on the $\ZZ$-homology groups (right):
%
% The existence of such a map $g\: K\to\partial K$ is a well-known impossibility.
% A brief argument can be given by discussion homology groups:

\begin{figure}[h!]
    \centering
$\begin{tikzcd}
	& K \\
	{\partial K} && {\partial K}
	\arrow["g|_{\partial K}", from=2-1, to=2-3]
	\arrow["i", hook, from=2-1, to=1-2]
	\arrow["g", from=1-2, to=2-3]
\end{tikzcd}$
\qquad
$\begin{tikzcd}
	& H_\bullet(K) \\
	{H_\bullet(\partial K)} && {H_\bullet(\partial K)}
	\arrow["(g|_{\partial K})_*", from=2-1, to=2-3]
	\arrow["i_*", from=2-1, to=1-2]
	\arrow["g_*", from=1-2, to=2-3]
\end{tikzcd}$
\end{figure}

\noindent
Since $g|_{\partial K}$ is homotopic to the identity, the arrow $(g|_{\partial K})_*$ is an isomorphism, and so must be the arrows above it. This is impossible because $$H_{d-1}(\partial K)=\ZZ\not=0=H_{d-1}(K).$$
\end{proof}

% \begin{lemma}
% Given the $d$-dimensional ball $B\subset\RR^d$ and a continuous map $f\: B\to\RR^d$ so that $0\not\in\partial B$ but the degree of the restricted map $f|_{\partial B}$ around the origin is non-zero. Then $0\in\Int(B)$.
% %
% \begin{proof}
% Suppose $0\not\in \Int(B)$. Then consider the following map $g\: \partial B\to\partial B$: for $x\in\partial B$ let $r$ be the ray emanating from $0\in\RR^d$ through $x$. Let $g(x)$ be the intersection of $r$ with $\partial B$.
% \end{proof}
% \end{lemma}

\section{Euler's homogeneous function theorem}

\begin{theorem}[Euler's homogeneous function theorem]
\label{res:Eulers_homogeneous_function_theorem}
Let $f\:\RR^n\to\RR$ be~a~homogeneous function of degree $d\ge1$ \ie\ $f(t\mathbf x)=t^d f(\mathbf x)$ for all $t\ge 0$.
Then
$$\sum_i x_i \frac{\partial f(\mathbf x)}{\partial x_i} = d\cdot f(\mathbf x).$$
\end{theorem}

\begin{proof}
Differentiate both sides of $f(t\mathbf x)=t^df(\mathbf x)$ \wrt\ $t$
%
% $$\frac{\partial}{\partial t}f(t\mathbf x)=\sum_i \frac{\partial f(t\mathbf x)}{\partial (tx_i)}\frac{\partial (tx_i)}{\partial t}=\sum_i x_i \frac{\partial f(t\mathbf x)}{\partial (tx_i)}.$$
% %
% and
% %
% $$\frac{\partial}{\partial t}(t^df(\mathbf x))=dt^{d-1} f(\mathbf x).$$
%
$$dt^{d-1} f(\mathbf x)=\frac{\partial}{\partial t}(t^df(\mathbf x))=\frac{\partial}{\partial t}f(t\mathbf x)=\sum_i \frac{\partial f(t\mathbf x)}{\partial (tx_i)}\frac{\partial (tx_i)}{\partial t}=\sum_i x_i \frac{\partial f(t\mathbf x)}{\partial (tx_i)}$$
and evaluate at $t=1$.
\end{proof}

\section{An alternative proof of \cref{res:expansion_main_result} using semi-definite optimization}
\label{sec:semi_definite_proof}

The following proof of \cref{res:expansion_main_result} does not address the equality case.

\begin{proof}
\cref{res:expansion_main_result} can be equivalently phrased as the claim that the following~program attains its optimum if we choose $q_i$ to be the skeleton of $P$:
$$
\begin{array}{rl}
    \max & \|q\|_\alpha \\
    %\sum_{i,j}\alpha_i\alpha_j\|q_i-q_j\|^2 \\
     %\text{s.t.} & \sum_i\alpha_i q_i=0 \\
     \text{s.t\rlap.} & \|q_i-q_j\|\le \|p_i-p_j\|,\quad\text{for all $ij\in E$}\\
     & q_1,...,q_n\in\RR^n
\end{array}
$$
%
%Let $\ell_{ij}:=\|p_i-p_j\|^2$ denote the edge lengths of $P$.
Since \mbox{$\|q\|_\alpha^2=\frac12\sum_{i,j}\alpha_i\alpha_j\|q_i-q_j\|^2$} $=\sum_i\alpha_i\|q_i\|^2 - \|\!\sum_i\alpha_iq_i\|^2$, we obtain the follo\-wing equivalent program:
$$
\begin{array}{rl}
    %\mathllap{e^*:=\;}
    \max & \sum_i\alpha_i\|q_i\|^2 =: e(q)\\
     \text{s.t\rlap.} & \sum_i\alpha_i q_i=0 \\
     & \|q_i-q_j\|\le \|p_i-p_j\|,\quad\text{for all $ij\in E$}\\
     & q_1,...,q_n\in\RR^n
\end{array}
$$
This particular 
program has been studied extensively (see \eg\ \cite{goring2011rotational,goring2008embedded,sun2006fastest}). 
It can~be re\-written as a semi-definite program (which we do not repeat here) with the~follow\-ing dual:
$$
\begin{array}{rl}
    %\mathllap{d^*:=\;}
    \min & \sum_{ij\in E} w_{ij}\|p_i-p_j\|^2 =: d(w) \\
     \text{s.t\rlap.} 
        & L_w-\diag(\alpha)+\mu\alpha\alpha\T\succeq0\\
        &w\ge 0,\mu\text{ free}
\end{array}
$$
where $L_w$ is the Laplace matrix of $G_P$ with edge weights $w$ (that is $L_{ij}=-w_{ij}$ and $L_{ii}=\sum_{j\not= i} w_{ij}$), $\diag(\alpha)$ is the diagonal matrix with $\alpha$ on its diagonal, and $X\succeq0$ asserts that $X$ is a positive semi-definite matrix.

Recall the following property of a dual program: if there are $q_i\in\RR^n$, $w\ge 0$ and $\mu\in\RR$ so that the primal and the dual program attain the same objective value, then we know that there is no duality gap and we found optimal solutions for both programs.
We now claim that such a choice can be made using $q_i:=p_i$, $w_{ij}:=M_{ij}$ (where $M$ is the Izmestiev matrix of $P$), and with a \mbox{value for $\mu$ to be determined la}\-ter. We first verify that the objective values agree: %\msays{where does the $1/2$ come from?}
%
%We claim that choosing $w_{ij}:=M_{ij}$, $M\in\RR^{n\times n}$ being the Izmestiev matrix of $P$, closes the duality gap, asserting that in both cases we found the optimal solution $e^*=d^*$.
%
\begin{align*}
d(M)&=\sum_{ij\in E}M_{ij}\|p_i-p_j\|^2
 =\tfrac12 \sum_{i,j}M_{ij}\|p_i-p_j\|^2
 \\&= \sum_i\Big(\sum_jM_{ij}\Big)\|p_i\|^2 - \sum_{i,j}M_{ij}\<p_i,p_j\>
 \\ &= \sum_i\alpha_i\|p_i\|^2 - \tr(\,\underbrace{MX_P}_{=0} X_P\T)
 \\&= \sum_i\alpha_i\|p_i\|^2=e(p),
\end{align*}
where $X_P\T:=(p_1,...,p_n)\in\RR^{n\x d}$,  $MX_P=0$ by \cref{res:Izmestiev} \itm4, as well as~$\sum_iM_{ij}$ $=\alpha_j$ by \cref{res:Wachspress_is_Izmestiev_row_sum}.

It only remains to verify that there exists $\mu\in\RR$ so that $L_M-\diag(\alpha)+\mu\alpha\alpha\T\succeq0$.
Set $D:=\diag(\alpha_1^{\smash{-1/2}},...,\alpha_n^{\smash{-1/2}})$ and observe that the matrices $X$ and $DXD$~have~the same signature.
It therefore suffices to verify
$$0\preceq D(L_M-\diag(\alpha)+\mu\alpha\alpha\T)D = DL_MD - \Id + \mu(D^{-1}\mathbf 1)(D^{-1}\mathbf 1)\T.$$
First we claim that $L_M-\diag(\alpha)=-M$. 
Since both sides agree on the off-diagonal, it suffices to compare their row sums. And in fact, since $L_M\mathbf1=0$~we~have~$(L_M-\diag(\alpha))\mathbf1=-\alpha=-M\mathbf1$.
Hence, $DL_MD-\Id$ has the same signature as~$-M$,\nls \ie\nls a unique negative eigenvalue, and one can check that the corresponding eigenvector is $D^{-1}\mathbf 1$. 
We see that the term $\mu(D^{-1}\mathbf 1)(D^{-1}\mathbf 1)\T$ just shifts this smallest eigenvalue of $DL_MD-\Id$ up or down, while not changing the other eigenvalues, and so we can choose $\mu$ large enough to make this eigenvalue positive. %\hfill$\square$
\end{proof}

The formulation of \cref{res:expansion_main_result} as a semi-definite program allows for a simultaneous reconstruction (\cf\  \cref{res:polytope_by_lenghts_and_alpha}) of both the polytope and its Izmestiev~matrix from only the edge-graph, the edge lengths and the Wachspress coordinates~of some interior point.
Since semi-definite programs can be solved in polynomial time, this approach is actually feasible in practice.
%in polynomial time (semi-definite programs can be solved in polynomial \cite{...}). % time, which shows that a recon\-struction as in \cref{res:polytope_by_lenghts_and_alpha} is actually feasible in practice.